\documentclass[reqno]{amsart}

\usepackage{amsmath}
\usepackage{amsthm}
\usepackage{amssymb}
\usepackage{mathtools}
\usepackage[hmarginratio=1:1]{geometry}
\usepackage{color}
\usepackage{graphicx}
\usepackage{enumitem}
\usepackage{mathrsfs}
\usepackage[T1]{fontenc}

\theoremstyle{definition}
\newtheorem{theorem}{Theorem}[section]

\newtheorem{proposition}[theorem]{Proposition}
\newtheorem{lemma}[theorem]{Lemma}
\newtheorem{corollary}[theorem]{Corollary}
\newtheorem{remark}[theorem]{Remark}

\numberwithin{equation}{section}
\numberwithin{table}{section}

\allowdisplaybreaks[3]

\newcommand{\NN}{\mathbb{N}}
\newcommand{\QQ}{\mathbb{Q}}
\newcommand{\FF}{\mathbb{F}}
\newcommand{\ZZ}{\mathbb{Z}}
\newcommand{\RR}{\mathbb{R}}
\newcommand{\CC}{\mathbb{C}}
\newcommand{\HH}{\mathbb{H}}
\newcommand{\trace}{\text{tr}}
\newcommand{\aut}{\text{Aut}}
\renewcommand{\pmod}[1]{\text{ (mod }#1)}
\newcommand{\legendre}[2]{\left(\frac{#1}{#2}\right)}
\newcommand{\sgn}{\text{sgn}}
\newcommand{\SL}{\text{SL}}
\newcommand{\holoproj}{\pi_{\text{hol}}}

\newcommand{\lcm}{\text{lcm}}
\newcommand{\ord}{\text{ord}}

\title[Moments of Hurwitz Class Numbers and Negative Bias Conjecture]{The bias conjecture for elliptic curves over finite fields and Hurwitz class numbers in arithmetic progressions}

\author{Ben Kane}
\address{The University of Hong Kong, Department of Mathematics, Pokfulam, Hong Kong}
\email{bkane@hku.hk}

\author{Sudhir Pujahari}
\address{School of Mathematics, National Institute of Science Education and Research, Bhimpur-Padanpur, Khurda, Odisha-752050, India.}
\email{spujahari@niser.ac.in/spujahari@gmail.com}
\author{Zichen Yang}
\address{Department of Mathematics, University of Hong Kong, Pokfulam, Hong Kong}
\email{zichenyang.math@gmail.com}

\date{\today}
\thanks{The research of the first author was supported by grants from the Research Grants Council of the Hong Kong SAR, China (project numbers HKU 17303618, 17314122, and 17305923). The second author is supported by Science and Engineering Research Board [SRG/2023/000930]. The third author was supported by the Dissertation Year Fellowship of the Unversity of Hong Kong. }
\begin{document}

\begin{abstract}
In this paper, we consider a version of the bias conjecture for second moments in the setting of elliptic curves over finite fields whose trace of Frobenius lies in an arbitrary fixed arithmetic progression. Contrary to the classical setting of reductions of one-parameter families over the rationals, where it is conjectured by Steven J. Miller that the bias is always negative, we prove that in our setting the bias is positive for a positive density of arithmetic progressions and negative for a positive density of arithmetic progressions. Along the way, we obtain explicit formulas for moments of traces of Frobenius of elliptic curves over finite fields in arithmetic progressions and related moments of Hurwitz class numbers in arithmetic progressions, the distribution of which are of independent interest.
\end{abstract}

\maketitle

\section{Introduction}
Let a prime number $p$ and integers $k\geq0,r,M\geq1,m\in\ZZ$ be given. We consider the elliptic curves $E$ over the finite field $\FF_{p^r}$ with $p^r$ elements whose trace of Frobenius $\trace(E)$ lies in the arithmetic progression $\trace(E)\equiv m\pmod{M}$. In \cite{kane2020distribution}, for $k$ even, the first two authors obtained an asymptotic formula for
\begin{equation}
\label{eqn::definemomenttrace}
S_{k,m,M}(p^r)\coloneqq\sum_{\substack{E/\FF_{p^r}\\\trace(E)\equiv m\pmod{M}}}\frac{\trace(E)^k}{|\aut(E)|}
\end{equation}
with a power-saving error term. It is natural to investigate the remaining terms of the asymptotic expansion. In this paper, we apply the theory of non-holomorphic modular forms and holomorphic projection to make every term in the asymptotic expansion explicit, modulo the contribution from Fourier coefficients of cusp forms. Specifically, by Theorem \ref{thm::maintheorem1}, Theorem \ref{thm::maintheorem2}, and Lemma \ref{thm::twomomentsrelation}, we obtain an explicit expansion of the following shape:
\begin{equation}
\label{eqn::asymptoticexpansion}
S_{k,m,M}(p^r)=~
\left(\parbox{13em}{a polynomial in $p$ with coefficients being sums of divisors of $p^t$ for some $t\in \ZZ$}\right)
~ + ~
\left(\parbox{15em}{a polynomial in $p$ with coefficients being (normalized) $p^r$-th Fourier coefficients of newforms}\right).
\end{equation}

The study of the expansion of this shape is motivated by the bias conjecture of Steven J. Miller. In a similar vein, for a one-parameter family of elliptic curves $E_t$, a main asymptotic formula for the second moment 
\[
\sum_{t\pmod{p}}\trace\left(E_t\right)^2
\]
was found by Michel \cite{Michel}, and Miller \cite{Mlr1,Mlr2} conjectured that, when the $j$-invariant of this family is non-constant, the second largest term which does not average to zero across the primes will have a negative average. This conjecture is related to the low-lying zeros in families of L-functions, which is part of the study of $n$-level densities by Katz and Sarnak \cite{KS1,KS2} and has connections with the zeros of elliptic curves (GRH) and Birch and Swinnerton-Dyer conjecture. For a more detail description, the reader may refer to \cite{MW,Asada,Mlr1,Mlr2}. Note that the reduction $E_t\pmod{p}$ (assuming good reduction) gives an elliptic curve over $\FF_p$, so the result of Michel and the conjecture of Miller are about a certain average of second moments of traces of Frobenius of elliptic curves over the finite field $\FF_p$, which is analogous to the setting for $S_{2,m,M}(p)$; the primary difference lies in the way in which elliptic curves over $\FF_p$ are selected in the sum. 

The splitting in the asymptotic formula of Michel is naturally given in terms of cohomology and is motivic in the sense that the terms of ``size'' $p^{\alpha}$ may be written (see \cite[(1.1)]{KazalickiNaskrecki}) in the form 
\[
\sum_{j}\beta_{\alpha,j}^r
\]
with $|\beta_{\alpha,j}|=p^\alpha$. The $\beta_{\alpha,j}$ in this setting naturally appear in a certain zeta function. Based on this, if the average over the terms of ``size'' $p^{\alpha}$ vanishes across primes $p$, then one ignores any possible contribution from these terms when considering the average over terms of ``size'' $p^{\alpha'}$ with $\alpha'<\alpha$. Mirroring this, as conjectured by Weil and proven by Deligne \cite{deligne1974conjecture}, one may write the $p$-th Fourier coefficient $a(p)$ of a newform of weight $\kappa$ on a congruence subgroup $\Gamma$ in the form 
\[
a(p)=\beta_p+\gamma_p,
\]
where $|\beta_p|=|\gamma_p|=p^{\frac{\kappa-1}{2}}$ are the complex roots of the Hecke polynomial. This gives a well-defined ``size'' of the Fourier coefficients appearing in the expansion (\ref{eqn::asymptoticexpansion}) when considering averages like those of Miller. We adopt this definition of ``size'' and henceforth only consider the contribution of $a(p)$ to the average over those terms with size $p^{\frac{\kappa-1}{2}}$. Analogous statements exist for the $p^r$-th coefficients written as homogeneous polynomials of degree $r$ in $\beta_p$ and $\gamma_p$ via the action of the Hecke operators. When considering sums of divisors of $p^t$ in the expansion (\ref{eqn::asymptoticexpansion}), we naturally consider the divisor $p^{\ell}$ to contribute to the terms with ``size'' $p^{\ell}$. This yields a well-defined definition of ``size'' for every term in the expansion.

Let $\mathscr{S}_{k,m,M,r}(p^j)$ denote the terms of size $p^j$, so that 
\[
S_{k,m,M}(p^r)=\sum_{j\in\frac{1}{2}\ZZ} p^j\sum_{\alpha\in\mathscr{S}_{k,m,M,r}(p^j)}\frac{\alpha}{p^j}.
\]
Note that each term $\frac{\alpha}{p^j}$ has $\big|\frac{\alpha}{p^j}\big|=1$. Using this splitting of the formula, we define the averages
\[
\mathcal{A}_{k,m,M,r,j}:=\lim_{X\to\infty}\frac{1}{\pi(X)}\sum_{p\leq X} \sum_{\alpha\in \mathscr{S}_{k,m,M,r}(p^j)}\frac{\alpha}{p^j},
\]
where $\pi(X)$ is the number of prime numbers $\leq X$. By \cite[Theorem 1.2]{kane2020distribution}, for $k$ even, we have
$$
\mathcal{A}_{k,m,M,r,j}=
\begin{dcases}
C_{k/2} &\text{ if }j=r(k/2+1),\\
0&\text{ if }j>r(k/2+1),
\end{dcases}
$$
where $C_n$ is the $n$-th Catalan number. Hence, to study the bias conjecture, we need to find the largest $j<k/2+1$ for which $\mathcal{A}_{k,m,M,r,j}\neq 0$. The case closely related to the conjecture of Miller is when $k=2$ and $r=1$. In this case, we find that the second largest term which does not average to zero is of size $p$ (as is expected in the conjecture of Miller), and further our formula (\ref{eqn::asymptoticexpansion}) allows us to give a simple expression of the bias, yielding an easily-checked equivalent condition for checking the positivity or negativity of the bias.

\begin{theorem}\label{thm:biaseval}
Let $M\geq1, m\in\ZZ$ be given integers. Then the following results hold:
\begin{enumerate}[leftmargin=8mm]
\item We have 
\begin{equation}\label{eqn:biaseval}
\mathcal{A}_{2,m,M,1,1}=\frac{1}{2\phi(M)}\Bigg(2\prod_{p\mid M}\frac{1-\delta_{p\nmid m}/(p^2-p)}{1+1/p}-\delta_{\gcd(m-1,M)=1}-\delta_{\gcd(m+1,M)=1}\Bigg),
\end{equation}
where the product runs over prime divisors $p\mid M$, $\phi$ is Euler's totient function, and throughout $\delta_{S}=1$ if a statement $S$ is true and $\delta_S=0$ otherwise. 
\item We have $\mathcal{A}_{2,m,M,1,1}=0$ if and only if $M=1$.
\item For $M>1$, the average $\mathcal{A}_{2,m,M,1,1}\neq 0$ is the second largest non-zero average. In particular, for $M>1$ the bias is positive (resp. negative) if and only if \eqref{eqn:biaseval} is positive (resp. negative).
\end{enumerate}
\end{theorem}

By studying \eqref{eqn:biaseval}, as a corollary we are able to show that both positive and negative biases occur with positive density. Moreover, in certain special cases we obtain simple conditions for the positivity or negativity of this bias.

\begin{theorem}[Theorem \ref{thm:bias}]
\label{thm:biasintro}
Let $M\geq1, m\in\ZZ$ be given integers. Then the following results hold:
\begin{enumerate}[leftmargin=8mm]
\item If $M$ is an odd prime power $M=p^t$, then $\mathcal{A}_{2,m,M,1,1}>0$ if and only if $m\equiv\pm1\pmod{p}$.
\item If $6\mid M$, then $\mathcal{A}_{2,m,M,1,1}<0$ if and only if either $\gcd(m-1,M)=1$ or $\gcd(m+1,M)=1$.
\item We have
\[
\liminf\limits_{X\to\infty}\frac{|\left\{(m,M)|1\leq m\leq M\leq X,\mathcal{A}_{2,m,M,1,1}>0\right\}|}{|\left\{(m,M)|1\leq m\leq M\leq X\right\}|}\geq\frac{1}{4}
\]
and
\[
\liminf\limits_{X\to\infty}\frac{|\left\{(m,M)|1\leq m\leq M\leq X,\mathcal{A}_{2,m,M,1,1}<0\right\}|}{|\left\{(m,M)|1\leq m\leq M\leq X\right\}|}\geq\frac{1}{2\pi^2}.
\]
\end{enumerate}
\end{theorem}
\begin{remark}
We remark that our estimate of the densities of positive bias and of negative bias appear to be far from the values expected by numerical evidence. By running a computer program, we observed that
$$
\frac{|\left\{(m,M)|1\leq m\leq M\leq 1000,\mathcal{A}_{2,m,M,1,1}>0\right\}|}{|\left\{(m,M)|1\leq m\leq M\leq 1000\right\}|}\approx0.44,
$$
and
$$
\frac{|\left\{(m,M)|1\leq m\leq M\leq 1000,\mathcal{A}_{2,m,M,1,1}<0\right\}|}{|\left\{(m,M)|1\leq m\leq M\leq 1000\right\}|}\approx0.56.
$$
\end{remark}

Although the original negative bias conjecture of Miller \cite{Mlr1,Mlr2} is restricted to the setting of the second moment of the reductions from one-parameter families to those over the finite field $\FF_p$ with $p$ elements, it is natural to consider the bias question of higher moments over other finite fields. In our setting, this generalization is readily available from our methods. Using the formula (\ref{eqn::asymptoticexpansion}), one may study moments of any order in any arithmetic progression over any finite field. 

As an example, we consider similar biases for elliptic curves $E/\FF_{p^2}$ whose traces lie in given arithmetic progressions. Here the answer is rather simple when $4\nmid M$ and $M\geq 3$. To state the result, let $M_{\operatorname{odd}}:=M/2^{\ord_2(M)}$ be the odd part of $M$, where $\ord_p(M)$ is the largest power of the prime $p$ dividing $M$.

\begin{theorem}[Theorem \ref{thm::lasttheorem}]
\label{thm:biasintror=2}
Let $M\geq1,m\in\ZZ$ be given integers such that $M\geq 3$ and $4\nmid M$. Then $\mathcal{A}_{2,m,M,2,3}$ is the largest non-vanishing error term and $\mathcal{A}_{2,m,M,2,3}<0$ if and only if $\gcd(m,M_{\operatorname{odd}})=1$.
\end{theorem}
\begin{remark}
Although we restrict our attention to second moments in this paper, in line with the original conjecture in the one-parameter family setting, Theorem \ref{thm::lasttheorem} indicates how the method extends to investigate the bias in different directions. In particular, Theorem \ref{thm::lasttheorem} extends the investigation to finite fields with $p^2$ elements, and Zindulka and the first author are using the method to determine the bias in the one-parameter Legendre family of elliptic curves for higher moments and finite fields of arbitrary size.
\end{remark}
Our results about moments of traces of Frobenius follow by investigating moments for Hurwitz class numbers and using a relation of Schoof \cite{schoof1987nonsingular}. Moments of these Hurwitz class numbers are of independent interest, so we next summarize our results about moments of Hurwitz class numbers.  For every positive integer $D$, we denote by $H(D)$ the $D$-th Hurwitz class number, which is the number of equivalence classes of positive-definite binary quadratic forms of discriminant $-D$, with the class containing a multiple of $x^2+y^2$ weighted by $\frac{1}{2}$ and the class containing a multiple of $x^2+xy+y^2$ weighted by $\frac{1}{3}$. By convention, we set $H(0)\coloneqq-\frac{1}{12}$ and set $H(D)\coloneqq0$ for any negative integer $D$.

The classical Kronecker--Hurwitz class number relation \cite{kronecker1860ueber,gierster1883ueber,hurwitz1885ueber} asserts that
$$
\sum_{t\in\ZZ}H(4p-t^2)=2p,
$$
for any prime number $p$. In \cite{brown2008elliptic}, Kronecker--Hurwitz class number relations in arithmetic progressions in the following were considered:
$$
H_{m,M}(n)\coloneqq\sum_{t\equiv m\pmod{M}}H(4n-t^2),
$$
where $m\in\ZZ$ and a prime number $M>1$ are fixed. For example, it was shown in \cite{brown2008elliptic} that for any odd prime number $p$, we have
\begin{equation}
\label{eqn::simplekroneckerhurwitz}
H_{m,2}(p)=
\begin{dcases}
\frac{4p-2}{3}&\text{ if }m\equiv0\pmod{2},\\
\frac{2p+2}{3}&\text{ if }m\equiv1\pmod{2}.
\end{dcases}
\end{equation}
Similar identities for prime numbers $M=3,5,7$ were proved or conjectured in the same paper. Later the conjectures were resolved in \cite{bringmann2019sums} and other authors have considered further generalizations \cite{mertens2016eichler,zindulka2024sums}. 

In this paper, we study the $k$-th moment of Hurwitz class numbers in an arithmetic progression, defined by
$$
H_{k,m,M}(n)\coloneqq\sum_{t\equiv m\pmod{M}}t^kH(4n-t^2),
$$
for any integers $n\geq0,k\geq0,m\in\ZZ,M\geq1$. In particular, the zeroth moment $H_{m,M}(n)=H_{0,m,M}(n)$ is exactly the Kronecker--Hurwitz class number relation in the arithmetic progression $m\pmod{M}$ considered above. Set
\begin{equation}
\label{eqn::extendedcatalan}
\mathcal{C}_k\coloneqq
\begin{dcases}
C_{\frac{k}{2}} &\text{if }k\geq2\text{ is even},\\
0 &\text{if }k\geq1\text{ is odd},
\end{dcases}
\end{equation}
and set
\begin{equation}
\label{eqn::sumsmallerdivisor}
\lambda_{k,m,M}(n)\coloneqq\sideset{}{^*}\sum_{\substack{t^2-s^2=4n\\t\equiv m\pmod{M}\\t\geq1,s\geq0}}\left(\frac{t-s}{2}\right)^{k+1}+\sideset{}{^*}\sum_{\substack{t^2-s^2=4n\\t\equiv-m\pmod{M}\\t\geq1,s\geq0}}(-1)^k\left(\frac{t-s}{2}\right)^{k+1},
\end{equation}
where the star superscript in the inner sum means that the summand is weighted by $1/2$ if $s=0$. Then we can relate higher moments of Hurwitz class numbers in an arithmetic progression to the zeroth moment.

\begin{theorem}[(\ref{eqn::highermomentexplicitformula}) and Theorem \ref{thm::coefficientformula}]
\label{thm::maintheorem1}
For any integers $n\geq0,k\geq0,m\in\ZZ,M\geq1$, we have
$$
H_{k,m,M}(n)=\mathcal{C}_kn^{\frac{k}{2}}H_{m,M}(n)+\sum_{\mu=0}^{[(k-1)/2]}\frac{k-2\mu+1}{k-\mu+1}\binom{k}{\mu}\big(a_{k-2\mu,m,M}(n)-\lambda_{k-2\mu,m,M}(n)\big)n^{\mu},
$$
where $\lambda_{k,m,M}(n)$ is the sum of smaller divisors defined in (\ref{eqn::sumsmallerdivisor}) and $a_{k,m,M}(n)$ is the $n$-th Fourier coefficient of a holomorphic cusp form for any integer $k\geq1$.
\end{theorem}

Therefore, it reduces the problem to the study of the zeroth moment. In our next result, we establish an explicit formula for the zeroth moment of Hurwitz class numbers in any arithmetic progression, using the theory of non-holomorphic modular forms and holomorphic projection. 

To state the results, we have to introduce the following notations. Let $1_N$ denote the trivial Dirichlet character of modulus $N$ throughout. Let $p$ be a prime number. For a non-zero integer $n$, we set $n_p\coloneqq p^{\ord_p(n)}$. For a Dirichlet character $\eta$, we denote by $M_{\eta}$ and $N_{\eta}$ the modulus and the conductor of the Dirichlet character, respectively. Then we denote by $\eta_p$ the $p$-part in the natural prime decomposition of $\eta$. If $p\nmid M_{\eta}$, then we have $\eta_p=1_1$. Furthermore we denote by $G(\eta)$ the Gauss sum associated to the Dirichlet character $\eta$, defined by
$$
G(\eta)=\sum_{x=0}^{M_{\eta}}\eta(x)e^{\frac{2\pi ix}{M_{\eta}}}.
$$
For any primitive Dirichlet character $\eta$, there exists a decomposition $\eta=\hat{\eta}\tilde{\eta}$, where $\hat{\eta}$ is a primitive quadratic character and $\tilde{\eta}$ is a primitive Dirichlet character such that $N_{\tilde{\eta}}=N_{\tilde{\eta}^2}$. Let $\eta^{\star}$ be the Dirichlet character $\tilde{\eta}\chi_{_{N_{\tilde{\eta}}}}$, where we denote by $\chi_D$ the Dirichlet character given by the Kronecker--Jacobi symbol $\legendre{\cdot}{D}$. Finally, for any integers $m\in\ZZ,M\geq1$ and any primitive Dirichlet character $\eta$ modulo $M$, we set
\begin{equation}
\label{eqn::drestriction}
S(\eta,m,M)\coloneqq\left\{d~\Big|~\frac{M^2}{N_{\eta}^2}~\Bigg|\begin{array}{l}\text{(1) $d_p\mid p^2m_p^2$ if $p\neq2\nmid N_{\eta}$, or if $p=2\nmid N_{\eta}$ and $M_2\geq2m_2$;}\\\text{(2) $d_pN_{\eta_p}^2\mid p^2m_p^2$ and $d_p\in\ZZ^2$ if $p\mid N_{\hat{\eta}}$;}\\\text{(3) $4_pd_p=m_p^2$ if $p\mid N_{\tilde{\eta}}$.}\end{array}\right\}.
\end{equation}

Suppose that $k$ is an integer and $\epsilon$ and $\eta$ are Dirichlet characters. We denote by $\sigma_{k,\epsilon,\eta}(n)$ the twisted divisor function, defined as
$$
\sigma_{k,\epsilon,\eta}(n)\coloneqq\sum_{d\mid n}\epsilon(n/d)\eta(d)d^k.
$$ 
For the zeroth moment of Hurwitz class numbers in an arithmetic progression, we prove the following formula, rewriting the zeroth moment in terms of a linear combination of divisor functions and the Fourier coefficient of a holomorphic cusp form.

\begin{theorem}[Theorem \ref{thm::quasimodularformula}]
\label{thm::maintheorem2}
For any integers $M\geq1,m\neq0$, there exists a holomorphic cusp form $g_{m,M}$ of weight $2$ whose $n$-th Fourier coefficient is denoted by $a_{0,m,M}(n)$ such that
$$
H_{m,M}(n)=\frac{2\zeta(2)}{ML(2,1_M)}\sum_{\substack{\eta\text{ primitive}\\N_{\eta}\mid M}}\sum_{\substack{d\geq1\\d\in S(\eta,m,M)}}a(\eta,d)\sigma_{1,\eta,\eta}\left(\frac{n}{d}\right)+a_{0,m,M}(n)-\lambda_{0,m,M}(n),
$$
with
\begin{equation}
\label{eqn::divisorcoefficient}
a(\eta,d)\coloneqq\frac{\epsilon_{N_{\eta_0}}^3\eta\big(\frac{4d}{\gcd(4d,m^2)}\big)\eta^{\star}(N_{\hat{\eta}})\Phi_2(\eta)G(\eta^{\star})}{\tilde{\eta}\big(\frac{m^2}{\gcd(4d,m^2)}\big)\hat{\eta}(D_{d,m})\phi(N_{\tilde{\eta}})G(\eta)}\sqrt{\frac{d}{DN_{\eta}}}\prod_{\substack{p\mid M\\p\nmid N_{\eta}}}\Psi_{d,m}(p)\prod_{\substack{p\mid N_{\hat{\eta}}\\d_pN_{\eta_p}^2=p^2m_p^2}}\frac{1}{1-p},
\end{equation}
where $D$ and $D_{d,m}$ are the square-free parts of $d$ and $m^2/\gcd(4d,m^2)$, respectively, $\Phi_2(\eta)$ is defined by
$$
\Phi_2(\eta)=
\begin{dcases}
(-i)\epsilon_{N_{\eta_0}}^2&\text{if }2\mid N_{\hat{\eta}}\text{ and }\eta_2\neq\chi_8,\\
1+\eta_2\left(1+N_{\eta_2}/4\right)&\text{if }2\mid N_{\tilde{\eta}},\\
1&\text{if }2\nmid N_{\eta}\text{ or }\eta_2=\chi_8.
\end{dcases}
$$
and $\Psi_{d,m}(p)$ is defined by
$$
\Psi_{d,m}(p)=
\begin{dcases}
1+\frac{\delta_{1<d_p<M_p^2}}{p^2}&\text{if }d_p<m_p^2\text{ and }d_p\in\ZZ^2,\\
-1+\frac{1}{p}&\text{if }d_p<m_p^2\text{ and }d_p\in p\ZZ^2,\\
1+\frac{\delta_{1<d_p<M_p^2}}{p^2}-\frac{\delta_{m_p<M_p}}{p^2-p}&\text{if }d_p=m_p^2,\\
-1+\frac{p^2+1}{p^2-p}&\text{if }d_p=pm_p^2,\\
1-\frac{p^2}{p^2-p}&\text{if }d_p=p^2m_p^2.
\end{dcases}
$$
\end{theorem}

As an immediate consequence, we can classify the pairs $(m,M)$ with the property $g_{m,M}=0$ in Theorem \ref{thm::maintheorem2}. Therefore, Kronecker-Hurwitz class numbers relations without involving cusp forms like (\ref{eqn::simplekroneckerhurwitz}) only exist for finitely many pairs $(m,M)$.

\begin{corollary}[Corollary \ref{thm::zerocuspidalpart}]
Suppose that $m,M$ are integers such that $1\leq m\leq M$. Then in Theorem \ref{thm::maintheorem2} we have $g_{m,M}=0$ if and only if $1\leq M\leq 5$ or $(m,M)=(2,8),(6,8)$.
\end{corollary}

The paper is organized as follows. In Section \ref{sec:HurwitzHighermoments}, we recall the basic properties of Rankin-Cohen brackets, holomorphic projections, and the class number generating function. As a result, we prove Theorem \ref{thm::maintheorem1} relating higher moments to the zeroth moment. In Section \ref{sec:holprojection}, we prove a general result for calculating the Eisenstein part of the holomorphic projection. Then it is used to prove a formula for the zeroth moment of the Hurwitz class numbers in Section \ref{sec:zerothmomentsHurwitz}. Finally, in Section \ref{sec:EllipticBias}, we study the bias given in Theorem \ref{thm:biasintro} and Theorem \ref{thm:biasintror=2} for elliptic curves over finite fields whose traces of Frobenius lie in any arithmetic progressions.

\section{Formulae for moments of Hurwitz class numbers}\label{sec:Hurwitz}

\subsection{Higher moments of Hurwitz class numbers}\label{sec:HurwitzHighermoments}

In this section, we rewrite the moments of Hurwitz class numbers in terms of the Fourier coefficients of Rankin-Cohen brackets, from which we obtain a recursive formula relates higher moments of Hurwitz class numbers to the zeroth moments. 

Set $\HH\coloneqq\{\tau\in\CC~|~\text{Im}(\tau)>0\}$. For any complex number $\tau\in\HH$, we write $\tau=u+iv$ for $u,v\in\RR$ and we set $q\coloneqq e^{2\pi i\tau}$ throughout. For smooth functions $F_1,F_2\colon\HH\to\CC$ and two half-integers $k_1,k_2\in\frac{1}{2}\ZZ$, the $\kappa$-th Rankin-Cohen bracket for an integer $\kappa\geq0$ is defined as
$$
[F_1,F_2]_{\kappa}\coloneqq\frac{1}{(2\pi i)^{\kappa}}\sum_{\mu=0}^{\kappa}(-1)^{\mu}\binom{\kappa+k_1-1}{\kappa-\mu}\binom{\kappa+k_2-1}{\mu}\frac{\partial^{\mu}F_1}{\partial\tau}\frac{\partial^{\kappa-\mu}F_2}{\partial\tau}.
$$
If the function $F_i$ satisfies the modularity of weight $k_i$ on a congruence subgroup $\Gamma_i$ for any $i=1,2$, then \cite[Theorem 7.1(a)]{cohen1975sums} implies that the Rankin-Cohen bracket $[F_1,F_2]_{\kappa}$ satisfies of weight $2\kappa+k_1+k_2$ on the intersection $\Gamma_1\cap\Gamma_2$. We are going to compute the Rankin-Cohen brackets of the class number generating function with unary theta functions.

The class number generating function $\mathcal{H}\colon\HH\to\CC$ is denoted by
$$
\mathcal{H}(\tau)\coloneqq\sum_{n=0}^{\infty}H(n)q^n.
$$
Setting
$$
\mathcal{H}^{+}(\tau)\coloneqq\mathcal{H}(\tau),~\mathcal{H^{-}}(\tau)\coloneqq\frac{1}{8\pi\sqrt{v}}+\frac{1}{4\sqrt{\pi}}\sum_{n=1}^{\infty}n\Gamma\left(-\frac{1}{2}, 4\pi n^2v\right)q^{-n^2},
$$
Zagier \cite[Th\'eor\`em 2]{zagier1974nombres} showed that the completion $\widehat{\mathcal{H}}(\tau)$ given by
$$
\widehat{\mathcal{H}}(\tau)=\mathcal{H}^{+}(\tau)+\mathcal{H^{-}}(\tau)
$$
is a harmonic Maass form of weight $\frac{3}{2}$ on the congruence subgroup $\Gamma_0(4)$.

Fix integers $\ell\geq0,M\geq1,m\in\ZZ$. The unary theta function $\theta_{\ell,m,M}\colon\HH\to\CC$ is defined as
\begin{equation}
\label{eqn::unarythetaseries}
\theta_{\ell,m,M}(\tau)\coloneqq\sum_{n\equiv m\pmod{M}}n^{\ell}q^{n^2}.
\end{equation}
For $\ell=0,1$, Shimura \cite[Proposition 2.1]{shimura1973modular} showed that $\theta_{\ell,m,M}$ is a modular form of weight $\frac{1}{2}+\ell$ on the congruence subgroup $\Gamma_{4M^2,M}$, where we set $\Gamma_{N,M}\coloneqq\Gamma_0(N)\cap\Gamma_1(M)$ for any integers $N,M\geq1$. 

By \cite[Theorem 7.1(b)]{cohen1975sums}, for $k\in\NN_0$ with $k\equiv \ell\pmod{2}$ we can compute the Fourier expansion of $[\mathcal{H},\theta_{\ell,m,M}]_{[\frac{k}{2}]}|_{U_4}$, where $[x]$ is the greatest integer $\leq x$ throughout, and the action of the operator $U_d$ on a modular form $f$ is given by
$$
f(\tau)=\sum a(n)q^n\longmapsto f|_{U_d}(\tau)=\sum a(nd)q^n.
$$
Following a formal calculation involving the moments of Hurwitz class numbers and the $n$-th Fourier coefficient $H^{\star}_{k,m,M}(n)$ of $[\mathcal{H},\theta_{\ell,m,M}]_{[\frac{k}{2}]}|_{U_4}$, we obtain the recursive formula
\begin{equation}
\label{eqn::recursiveformula1}
H^{\star}_{k,m,M}(n)=\binom{k}{[k/2]}\sum_{\mu=0}^{[k/2]}(-1)^{\mu}\binom{k-\mu}{\mu}n^{\mu}H_{k-2\mu,m,M}(n).
\end{equation}

Using an idea of Mertens \cite{mertens2016eichler}, we can compute the Fourier coefficient $H^{\star}_{k,m,M}(n)$ in another way. Let $\holoproj$ denote the holomorphic projection defined in \cite{gross1986heegner}. Since $\holoproj$ is linear and preserves holomorphic functions, we see that
\begin{equation}
\label{eqn::completedbracket}
[\mathcal{H},\theta_{\ell,m,M}]_{[\frac{k}{2}]}|_{U_4}=\holoproj\left([\widehat{\mathcal{H}},\theta_{\ell,m,M}]_{[\frac{k}{2}]}|_{U_4}\right)-\holoproj\left([\mathcal{H}^{-},\theta_{\ell,m,M}]_{[\frac{k}{2}]}|_{U_4}\right).
\end{equation}

For the first piece on the right-hand side of (\ref{eqn::completedbracket}), we have the following lemma.

\begin{lemma}
\label{thm::harmonicpart}
For any integer $k\geq0,M\geq1,m\in\ZZ$ and $\ell\in\{0,1\}$ with $k\equiv \ell\pmod{2}$, the holomorphic projection $\holoproj([\widehat{\mathcal{H}},\theta_{\ell,m,M}]_{[\frac{k}{2}]}|_{U_4})$ is a holomorphic cusp form of weight $k+2$ on the congruence subgroup $\Gamma_{4M^2,M}$ if $k>0$, and is a quasi-modular form of weight $2$ on the congruence subgroup $\Gamma_{4M^2,M}$ if $k=0$.
\end{lemma}
\begin{proof}
See \cite[Theorem 1.2]{mertens2016eichler}, \cite[Proposition 4.6]{ono2021distribution}, and \cite[Lemma 3.1]{bringmann2021odd}.
\end{proof}

The second piece on the right-hand side of (\ref{eqn::completedbracket}) is determined in the next lemma.

\begin{lemma}
\label{thm::nonholomorphicpart}
For any integer $k\geq0,m\in\ZZ,M\geq1$ and $\ell\in\{0,1\}$ with $k\equiv \ell\pmod{2}$, we have
$$
\holoproj([\mathcal{H}^{-},\theta_{\ell,m,M}]_{[\frac{k}{2}]}|_{U_4})=\binom{k}{[k/2]}\Lambda_{k,m,M}(\tau),
$$
where for any integers $k\geq0,M\geq1,m\in\ZZ$, we have
$$
\Lambda_{k,m,M}(\tau)\coloneqq\sum_{n\geq1}\lambda_{k,m,M}(n)q^n,
$$
with $\lambda_{k,m,M}(n)$ defined in (\ref{eqn::sumsmallerdivisor}).
\end{lemma}
\begin{proof}
For any even integer $k$, it was essentially shown in \cite[Theorem 1.2]{mertens2016eichler} and a correction of the leading factor was given in \cite[Remark of Lemma 2.7]{kane2020distribution}. For any odd integer $k$, it was calculated in \cite[Lemma 3.2]{bringmann2021odd}.
\end{proof}

Plugging Lemma \ref{thm::harmonicpart}, Lemma \ref{thm::nonholomorphicpart}, and (\ref{eqn::completedbracket}) into (\ref{eqn::recursiveformula1}), we see that for any integer $k\geq1$, we have
\begin{equation}
\label{eqn::recursiveformula2}
H_{k,m,M}(n)=a_{k,m,M}(n)-\lambda_{k,m,M}(n)-\sum_{\mu=1}^{[k/2]}(-1)^{\mu}\binom{k-\mu}{\mu}n^{\mu}H_{k-2\mu,m,M}(n),
\end{equation}
where we denote by $a_{k,m,M}(n)$ the $n$-th Fourier coefficient of 
$$
\binom{k}{[k/2]}^{-1}\holoproj([\widehat{\mathcal{H}},\theta_{\ell,m,M}]_{[\frac{k}{2}]}|_{U_4}),
$$
for any integer $k\geq1$. Next we discuss the case when $k=0$. Clearly the formula (\ref{eqn::recursiveformula2}) holds when $k=0$. However, if we set
\begin{equation}
\label{eqn::k=0notation}
\widehat{\mathcal{H}}_{m,M}\coloneqq[\widehat{\mathcal{H}},\theta_{0,m,M}]_{0}|_{U_4}=(\widehat{\mathcal{H}}\theta_{0,m,M})|_{U_4},
\end{equation}
then by Lemma \ref{thm::harmonicpart}, the holomorphic projection $\widehat{\mathcal{H}}_{m,M}$ is not necessarily a holomorphic cusp form. To avoid confusion, we denote by $\widehat{H}_{m,M}(n)$ its $n$-th Fourier coefficient and the notation $a_{0,m,M}(n)$ is reserved for the $n$-th Fourier coefficient of its cuspidal part $g_{m,M}$ given in Theorem \ref{thm::maintheorem2}. Therefore, if $k=0$, then we have
\begin{equation}
\label{eqn::k=0}
H_{m,M}(n)=\widehat{H}_{m,M}(n)-\lambda_{k,m,M}(n).
\end{equation}

Iterating the recursive formula (\ref{eqn::recursiveformula2}), we see that there exist integers $T(k,\mu)$ for any integer $k\geq1$ and any integer $0\leq\mu\leq[\frac{k-1}{2}]$ such that
\begin{equation}
\label{eqn::highermomentexplicitformula}
H_{k,m,M}(n)=\mathcal{C}_kn^{\frac{k}{2}}H_{m,M}(n)+\sum_{\mu=0}^{[(k-1)/2]}T(k,\mu)\big(a_{k-2\mu,m,M}(n)-\lambda_{k-2\mu,m,M}(n)\big)n^{\mu},
\end{equation}
where $\mathcal{C}_k$ is defined in (\ref{eqn::extendedcatalan}). Comparing (\ref{eqn::recursiveformula2}) with (\ref{eqn::highermomentexplicitformula}), we see that $T(k,\mu)$ satisfies the following recursive relation:
\begin{equation}
\label{eqn::tkmu}
T(k,0)=1,~T(k,\mu)=-\sum_{m=1}^{\mu}(-1)^m\binom{k-m}{m}T(k-2m,\mu-m),
\end{equation}
for any integer $k\geq1$ and any integer $1\leq\mu\leq[\frac{k-1}{2}]$. In fact, there is a simple formula for the integers $T(k,\mu)$.

\begin{theorem}
\label{thm::coefficientformula}
For any integer $k\geq1$ and any integer $0\leq\mu\leq[\frac{k-1}{2}]$, we have
$$
T(k,\mu)=\frac{k-2\mu+1}{k-\mu+1}\binom{k}{\mu}.
$$
\end{theorem}
\begin{proof}
First, it is easy to prove the formula for $\mu=0$. Next, we prove the formula by induction on $\mu$. By the inductive hypothesis, it is equivalent to proving that
\begin{align*}
\frac{k-2\mu+1}{k-\mu+1}\binom{k}{\mu}=&-\sum_{m=1}^{\mu}(-1)^m\frac{k-2\mu+1}{k-\mu-m+1}\binom{k-m}{m}\binom{k-2m}{\mu-m}\\
=&-\sum_{m=1}^{\mu}(-1)^m(k-2\mu+1)\frac{(k-m)!}{m!(\mu-m)!(k-\mu-m+1)!}.
\end{align*}
for any integer $1\leq\mu\leq[\frac{k-1}{2}]$. Multiplying both sides by $\frac{\mu!}{k-2\mu+1}$ and moving the sum to the left-hand side, it is equivalent to showing that
\begin{equation}
\label{eqn::identity}
\sum_{m=0}^{\mu}(-1)^m\binom{\mu}{m}\frac{(k-m)!}{(k-\mu-m+1)!}=0,
\end{equation}
for any integer $\mu\geq1$ and any integer $k\geq2\mu+1$. Let $x$ be a real number such that $|x|<1$. Taking the $(\mu-1)$-th derivatives of both sides in the following equation
$$
\sum_{k=\mu}^{\infty}x^k=\frac{1}{1-x}-\sum_{k=0}^{\mu-1}x^k,
$$
we see that
$$
\sum_{k=\mu}^{\infty}\frac{k!}{(k-\mu+1)!}x^{k-\mu+1}=\frac{(\mu-1)!\left(1-(1-x)^{\mu}\right)}{(1-x)^{\mu}}.
$$
It follows that
$$
f(x)\coloneqq\sum_{k=\mu}^{\infty}\frac{k!}{(k-\mu+1)!}x^k=\frac{(\mu-1)!x^{\mu-1}\left(1-(1-x)^{\mu}\right)}{(1-x)^{\mu}}.
$$
Therefore, we have
$$
(\mu-1)!x^{\mu-1}\left(1-(1-x)^{\mu}\right)=(1-x)^{\mu}f(x)=\sum_{k=\mu}^{\infty}\sum_{m=0}^{\min(\mu,k-\mu)}(-1)^m\binom{\mu}{m}\frac{(k-m)!}{(k-\mu-m+1)!}x^k.
$$
Since this equality holds for any real number $|x|<1$, we see that (\ref{eqn::identity}) holds for any integer $\mu\geq1$ and any integer $k\geq2\mu+1$.
\end{proof}

\subsection{Holomorphic projection of $\widehat{\mathcal{H}}_{m,M}$}\label{sec:holprojection}

Next we study the Eisenstein part of the holomorphic projections of $\widehat{\mathcal{H}}_{m,M}$, which will be used to obtain an explicit formula for the zeroth moment of Hurwitz class numbers. 

Let $f\colon\HH\to\CC$ be a function satisfying modularity of weight $k\in\frac{1}{2}\ZZ$. Then for any cusp $\rho=\frac{a}{c}\in\QQ\cup\{\infty\}$ such that $\gcd(a,c)=1$, we define the growth of $f$ at the cusp $\rho$ to be the limit
$$
\nu_{\rho}(f)\coloneqq\lim\limits_{v\to\infty}(c\tau+d)^{-k}f\left(\frac{a\tau+b}{c\tau+d}\right),
$$
whenever it exists, where the integers $b,d\in\ZZ$ are chosen so that $\begin{psmallmatrix}a&b\\c&d\end{psmallmatrix}\in\SL_2(\ZZ)$.

We shall use harmonic and quasi-modular Eisenstein series of weight $2$ that are constructed as follows. For any integer $k\geq0$, let $B_{k,\chi}$ denote the $k$-th generalized Bernoulli number associated to the Dirichlet character $\chi$ defined by the relation
$$
\sum_{k=0}^{\infty}\frac{B_{k,\chi}}{k!}x^k=\sum_{n=1}^{M_{\chi}}\frac{\chi(n)xe^{nx}}{e^{M_{\chi}x}-1}.
$$
For primitive Dirichlet characters $\epsilon$ and $\eta$ such that $\epsilon(-1)\eta(-1)=1$, let $\omega$ denote the primitive Dirichlet character inducing $\epsilon\cdot\overline{\eta}$. The quasi-modular Eisenstein series of weight $2$ associated to the Dirichlet characters $\epsilon$ and $\eta$ is defined by 
\begin{equation}
\label{eqn::quasimodulareisenstein}
E_{2,\epsilon,\eta}(\tau)\coloneqq\epsilon(0)-\frac{4\cdot G(\overline{\eta})N_{\omega}^2}{B_{2,\overline{\omega}}G(\omega)N_{\eta}^2}\prod_{p\mid\lcm(N_{\epsilon},N_{\eta})}\frac{p^2}{p^2-\omega(p)}\sum_{n\geq1}\sigma_{1,\epsilon,\eta}(n)q^n.
\end{equation}

The well-known harmonic Eisenstein series of weight $2$ is given by
\begin{equation}
\label{eqn::E2expansion}
\widehat{E}_2\coloneqq E_{2,1_1,1_1}-\frac{3}{\pi v},
\end{equation}
which satisfies weight $2$ modularity on $\SL_2(\ZZ)$. By (\ref{eqn::quasimodulareisenstein}) and (\ref{eqn::E2expansion}), it is clear that $\nu_{\rho}(\widehat{E}_2)=1$ towards any cusp $\rho$. More generally, fixing integers $N,M\geq1$ such that $M\mid N$, we can apply the construction in \cite[Chapter VII, (25)]{schoeneberg2012elliptic} and the trace operator from $\Gamma(N)$ to $\Gamma_{N,M}$ to construct harmonic Eisenstein series. Specifically, for any cusp $\rho$ of $\Gamma_{N,M}$, there exists a harmonic Eisenstein series $\widetilde{E}_{2,N,M}^{\rho}$ of weight $2$ on the congruence subgroup $\Gamma_{N,M}$ with two properties:
\begin{enumerate}[leftmargin=8mm]
\item The Fourier expansion of $\widetilde{E}_{2,N,M}^{\rho}$ is given by
\begin{equation}
\label{eqn::E2NMexpansion}
\widetilde{E}_{2,N,M}^{\rho}(\tau)=\text{holomorphic part}-\frac{3\alpha}{\pi v},
\end{equation}
for a rational number $\alpha\in\QQ$. The number $\alpha$ can be determined explicitly but we only need the existence in this paper.
\item The harmonic Eisenstein series $\widetilde{E}_{2,N,M}^{\rho}$ is non-vanishing towards the cusp $\rho$ and vanishes towards the other cusps.
\end{enumerate}

Before proving the main result in this section, we introduce the following decomposition relating modular forms on the congruence subgroup $\Gamma_{N,M}$ to modular forms on the congruence subgroup $\Gamma_0(N)$. Let $f\colon\HH\to\CC$ be a function satisfying modularity of weight $k$ on the congruence subgroup $\Gamma_{N,M}$ with integers $N,M\geq1$ such that $M\mid N$. For a Dirichlet character $\psi$ modulo $M$, we set
\begin{equation}
\label{eqn::explicitchodecomposition}
f_{\psi}\coloneqq\frac{1}{\phi(M)}\sum_{d\in(\ZZ/M\ZZ)^{\times}}\overline{\psi}(d)f|_{\gamma_d},
\end{equation}
where $\gamma_d$ is a matrix in $\Gamma_0(N)$ such that its lower-right entry is congruent to $d$ modulo $M$ for each integer $d\in(\ZZ/M\ZZ)^{\times}$.

\begin{lemma}
\label{thm::chodecomposition}
Fix integers $N,M\geq1$ such that $M\mid N$. We have a direct sum decomposition
$$
\mathcal{E}_2(\Gamma_{N,M})=\bigoplus_{\psi}\mathcal{E}_2(\Gamma_0(N),\psi),
$$
where the sum runs over the group of Dirichlet characters modulo $M$. Moreover, the splitting for a holomorphic Eisenstein series $E\in\mathcal{E}_2(\Gamma_{N,M})$ is given by 
$$
E=\sum_{\psi}E_{\psi}
$$
with $E_{\psi}\in\mathcal{E}_2(\Gamma_0(N),\psi)$ for each Dirichlet character $\psi$ modulo $M$.
\end{lemma}
\begin{proof}
See \cite[Theorem 2.5]{cho2018number}.
\end{proof}

We shall use the following proposition to obtain an explicit formula for the zeroth moment of Hurwitz class numbers, which can be considered as a generalization of \cite[Theorem 1]{aygin2022projections} to quasi-modular Eisenstein series.

\begin{proposition}
\label{thm::quasimodularformulageneral}
Set $f=\widehat{\mathcal{H}}_{m,M}$ defined in (\ref{eqn::k=0notation}) for integers $m\in\ZZ,M\geq1$. Suppose that $f=\sum_{\psi}f_{\psi}$ be the splitting in Lemma \ref{thm::chodecomposition}. Then there exists a quasi-modular Eisenstein series $E$ of weight $2$ and a holomorphic cusp form of weight $2$ on the congruence subgroup $\Gamma_{4M^2,M}$ such that
$$
\holoproj(f)=E+g.
$$
Moreover, we have
\begin{equation}
\label{eqn::eisensteinformula}
E(\tau)=\sum_{\psi}\sum_{\substack{\epsilon,\eta\text{ primitive}\\\epsilon\eta=\psi}}\sum_{d\mid 4M^2/N_{\epsilon}N_{\eta}}a_{f_{\psi}}(\epsilon,\eta,d)E_{2,\epsilon,\eta}(d\tau),
\end{equation}
where the coefficient $a_f(\epsilon,\eta,d)$ is given by
$$
a_f(\epsilon,\eta,d)\coloneqq\prod_{p\mid2M}\frac{p^2}{p^2-\epsilon(p)\overline{\eta}(p)}\sum_{c\mid 4M^2/N_{\epsilon}N_{\eta}}\mathcal{R}_{\epsilon,\eta}(d,c)\mathcal{S}_{4M^2/N_{\epsilon}N_{\eta},\epsilon,\eta}(d,c)\sum_{a\in(\ZZ/cN_{\eta}\ZZ)^{\times}}\frac{\eta(a)\nu_{\frac{a}{cN_{\eta}}}(f)}{\phi(cN_{\eta})},
$$
with
$$
\mathcal{R}_{\epsilon,\eta}(d,c)=\epsilon\left(\frac{-d}{\gcd(c,d)}\right)\overline{\eta}\left(\frac{c}{\gcd(c,d)}\right)\frac{\gcd(c,d)^2}{c^2},
$$
and
$$
\mathcal{S}_{N,\epsilon,\eta}(d,c)=\mu\left(\frac{cd}{\gcd(c,d)^2}\right)\prod_{\substack{p\mid\gcd(c,d)\\\ord_p(c)=\ord_p(d)\\\ord_p(d)<\ord_p(N)}}\left(\frac{p^2+\epsilon(p)\overline{\eta}(p)}{p^2}\right).
$$
\end{proposition}
\begin{proof}
Let $C(M)$ denote a complete set of representatives of cusps of $\Gamma_{4M^2,M}$. Set
$$
\widetilde{E}\coloneqq\sum_{\rho\in C(M)}\frac{\nu_{\rho}(f)}{\nu_{\rho}\big(\widetilde{E}_{2,4M^2,M}^{\rho}\big)}\widetilde{E}_{2,4M^2,M}^{\rho}.
$$
Then it is clear that the function $f-\widetilde{E}$ vanishes towards any cusp $\rho$. Therefore, applying the holomorphic projection, we see that
$$
g\coloneqq\holoproj(f-\widetilde{E})=\holoproj(f)-\holoproj(\widetilde{E})
$$
is a holomorphic cusp form of weight $2$ on the congruence subgroup $\Gamma_{4M^2,M}$ by \cite[Corollary 3.2]{bringmann2023formulas}. It remains to show that $E\coloneqq\holoproj(\widetilde{E})$ is a quasi-modular Eisenstein series satisfying (\ref{eqn::eisensteinformula}). By (\ref{eqn::E2expansion}) and (\ref{eqn::E2NMexpansion}), there exists a constant $\alpha\in\QQ$ such that $E_0\coloneqq\widetilde{E}-\alpha\widehat{E}_2$ is a holomorphic Eisenstein series satisfying the modularity of weight $2$ on the congruence subgroup $\Gamma_{4M^2,M}$. Applying Lemma \ref{thm::chodecomposition}, we have
$$
E_0=\sum_{\psi}E_{0,\psi}.
$$
Moreover, by (\ref{eqn::explicitchodecomposition}), we have
\begin{equation}
\label{eqn::E0growth}
\nu_{\rho}(E_{0,\psi})=\nu_{\rho}(f_{\psi})-\alpha\delta_{\psi=1_M},
\end{equation}
for any cusp $\rho$ and any Dirichlet character $\psi$ modulo $M$. Applying \cite[Theorem 1]{aygin2022projections} with (\ref{eqn::E0growth}), we see that
$$
E_0(\tau)=\sum_{\psi}\sum_{\substack{\epsilon,\eta\text{ primitive}\\\epsilon\eta=\psi}}\sum_{d\mid4M^2/N_{\epsilon}N_{\eta}}a_{f_{\psi}}(\epsilon,\eta,d)E_{2,\epsilon,\eta}(d\tau)-\sum_{\eta\text{ primitive}}\sum_{d\mid4M^2/N_{\eta}^2}a_{\alpha}(\overline{\eta},\eta,d)E_{2,\overline{\eta},\eta}(d\tau).
$$
If $\eta\neq1_1$, then it follows from a straightforward calculation that $a_{\alpha}(\overline{\eta},\eta,d)=0$. If $\eta=1_1$, we have
\begin{align*}
a_{\alpha}(1_1,1_1,d)=&~\alpha\prod_{p\mid2M}\left(\frac{p^2}{p^2-1}\right)\sum_{c\mid4M^2}\mathcal{R}_{1_1,1_1}(d,c)\mathcal{S}_{4M^2,1_1,1_1}(d,c)\\
=&~\alpha\prod_{p\mid2M}\left(\frac{p^2}{p^2-1}\right)\sum_{c\mid4M^2}\frac{\gcd(c,d)^2}{c^2}\mu\left(\frac{cd}{\gcd(c,d)^2}\right)\prod_{\substack{p\mid\gcd(c,d)\\\ord_p(c)=\ord_p(d)\\\ord_p(d)<\ord_p(4M^2)}}\left(\frac{p^2+1}{p^2}\right).
\end{align*}
Using multiplicativity to evaluate the sum, we see that $a_{\alpha}(1_1,1_1,d)=\delta_{d=1}\alpha$. Therefore, we have
$$
E_0(\tau)=\sum_{\psi}\sum_{\substack{\epsilon,\eta\text{ primitive}\\\epsilon\eta=\psi}}\sum_{d\mid4M^2/N_{\epsilon}N_{\eta}}a_{f_{\psi}}(\epsilon,\eta,d)E_{2,\epsilon,\eta}(d\tau)-\alpha E_{2,1_1,1_1}(\tau).
$$
Hence, we have
$$
E=\holoproj(E_0)+\alpha\holoproj(\widehat{E}_2)=E_0+\alpha E_2=\sum_{\psi}\sum_{\substack{\epsilon,\eta\text{ primitive}\\\epsilon\eta=\psi}}\sum_{d\mid N/N_{\epsilon}N_{\eta}}a_{f_{\psi}}(\epsilon,\eta,d)E_{2,\epsilon,\eta}(d\tau),
$$
as desired.
\end{proof}

\subsection{The zeroth moment of Hurwitz class numbers}\label{sec:zerothmomentsHurwitz}

Recall that in (\ref{eqn::k=0}), we have
\begin{equation}
\label{eqn::momentquasimodular}
H_{m,M}(n)=\widehat{H}_{m,M}(n)-\lambda_{0,m,M}(n),
\end{equation}
where $\widehat{H}_{m,M}(n)$ is the $n$-th Fourier coefficient of $\holoproj(\widehat{\mathcal{H}}_{m,M})$, and $\lambda_{0,m,M}(n)$ is defined in (\ref{eqn::sumsmallerdivisor}). In this section, we apply Proposition \ref{thm::quasimodularformulageneral} to prove an explicit formula for the zeroth moment of Hurwitz class numbers. This requires us to compute the growth of certain modular forms towards cusps, which is done in the subsequent lemmas.

\begin{lemma}
\label{thm::hurwitzgrowth}
For any integers $a,c\in\ZZ$ such that $c\geq0$ and $\gcd(a,c)=1$, we have
$$
\nu_{\frac{a}{c}}(\widehat{\mathcal{H}})=
\begin{dcases}
-\frac{1}{12}\legendre{c}{a}\varepsilon_a&\text{if }c\equiv0\pmod{4},\\
0&\text{if }c\equiv2\pmod{4},\\
-\frac{1+i}{96}\legendre{a}{c}\varepsilon_c^3&\text{if }c\equiv\pm1\pmod{4},
\end{dcases}
$$
where for an odd integer $d$, we set 
$$
\epsilon_d\coloneqq
\begin{dcases}
1&\text{if }d\equiv1\pmod{4},\\
i&\text{if }d\equiv3\pmod{4}.
\end{dcases}
$$
\end{lemma}
\begin{proof}
See \cite[Proposition 6.7]{bringmann2017harmonic}.
\end{proof}

\begin{lemma}
\label{thm::thetagrowth}
Fix integers $M\geq1,m\in\ZZ$. For any integers $a,c\in\ZZ$ such that $c\geq0$ and $\gcd(a,c)=1$, we have
$$
\nu_{\frac{a}{c}}(\theta_{0,m,M})=
\begin{dcases}
\delta_{M\mid m}&\text{if }c=0,\\
\frac{(-i)^{\frac{1}{2}}}{M\sqrt{2c}}\sum_{x=0}^{c-1}e^{2\pi i\frac{a(Mx+m)^2}{c}}&\text{if }c\neq0.
\end{dcases}
$$
\end{lemma}
\begin{proof}
See \cite[Section 2]{shimura1973modular}.
\end{proof}

Suppose that $a,b,c\in\ZZ$ such that $c\geq1$. Then, the generalized quadratic Gauss sum $G(a,b,c)$ is defined as follows,
$$
G(a,b,c)\coloneqq\sum_{x=0}^{c-1}e^{2\pi i\frac{ax^2+bx}{c}}.
$$
Therefore, to study the growth of unary theta functions at each cusp, we have to evaluate generalized quadratic Gauss sums. Now we give formulae for generalized quadratic Gauss sums.

\begin{proposition}
\label{thm::gausssum}
Let $a,b,c\in\ZZ$ such that $c\geq1$. Then we have the following results:
\begin{enumerate}[leftmargin=8mm]
\item We have
$$
G(a,b,c)=
\begin{dcases}
\gcd(a,c)\cdot G\left(\frac{a}{\gcd(a,c)},\frac{b}{\gcd(a,c)},\frac{c}{\gcd(a,c)}\right)&\text{ if }\gcd(a,c)\mid b,\\
0&\text{ if }\gcd(a,c)\nmid b.
\end{dcases}
$$
\item If $\gcd(a,c)=1$ and $c\equiv1\pmod{2}$, then let $k\in\ZZ$ be an integer satisfying $4ak\equiv1\pmod{c}$. Then we have
$$
G(a,b,c)=e^{-\frac{2\pi ikb^2}{c}}\varepsilon_c\legendre{a}{c}\sqrt{c}.
$$
\item If $\gcd(a,c)=1$ and $c\equiv2\pmod{4}$, we set $c_0\coloneqq\frac{c}{2}$ and take an integer $k\in\ZZ$ such that $8ak\equiv1\pmod{c_0}$. Then we have
$$
G(a,b,c)=
\begin{dcases}
e^{-\frac{2\pi ikb^2}{c_0}}\varepsilon_{c_0}\legendre{2a}{c_0}\sqrt{2c}&\text{ if }b\text{ is odd},\\
0&\text{ if }b\text{ is even}.
\end{dcases}
$$
\item If $\gcd(a,c)=1$ and $c\equiv0\pmod{4}$, we take an integer $k\in\ZZ$ such that $ak\equiv1\pmod{c}$. Then we have
$$
G(a,b,c)=
\begin{dcases}
e^{-\frac{2\pi ikb^2}{4c}}\varepsilon_{a}^3\legendre{c}{a}(1+i)\sqrt{c}&\text{ if }b\text{ even},\\
0&\text{ if }b\text{ odd}.
\end{dcases}
$$
\end{enumerate}
\end{proposition}
\begin{proof}
This is well-known, and it may be found in this form in \cite[Lemma 2.6]{bringmann2022conjectures}.
\end{proof}

\begin{lemma}
\label{thm::u4growth}
Suppose that $f$ is a modular form of weight $2$. For any integers $a,c\in\ZZ$ such that $\gcd(a,c)=1$, we have
$$
\nu_{\frac{a}{c}}(f|_{U_4})=\frac{1}{4}\sum_{0\leq j\leq3}g_j^2\cdot\nu_{\frac{(a+cj)/g_j}{4c/g_j}}(f),
$$
where for any $0\leq j\leq3$, we set $g_j\coloneqq\gcd(a+cj,4c)$.
\end{lemma}
\begin{proof}
We can write
$$
f|_{U_4}(\tau)=\sum_{0\leq j\leq3}f|_{\begin{psmallmatrix}1 & j\\0 & 4\end{psmallmatrix}}(\tau).
$$
For each $0\leq j\leq3$, we have
$$
\begin{pmatrix}1 & j\\0 & 4\end{pmatrix}\begin{pmatrix}a & b\\c & d\end{pmatrix}=\gamma_j\begin{pmatrix}g_j & *\\0 & \frac{4}{g_j}\end{pmatrix},
$$
where $\gamma_j$ is a matrix of the form
$$
\gamma_j=
\begin{pmatrix}
\frac{a+cj}{g_j} & *\\
\frac{4c}{g_j} & *
\end{pmatrix}\in\SL_2(\ZZ).
$$
Therefore, we can evaluate the limit as follows,
$$
\lim\limits_{v\to\infty}f|_{U_4}|_{\gamma}(\tau)=\lim\limits_{v\to\infty}\sum_{j\pmod{4}}f|_{\gamma_j}|_{\begin{psmallmatrix}g_j & *\\0 & \frac{4}{g_j}\end{psmallmatrix}}(\tau)=\frac{1}{4}\sum_{j\pmod{4}}g_j^2\lim\limits_{v\to\infty}f|_{\gamma_j}\left(\frac{g_j^2\tau+*}{4}\right).
$$
Taking the limits of both sides, we obtain the desired result.
\end{proof}

Using Lemma \ref{thm::hurwitzgrowth}, Lemma \ref{thm::thetagrowth}, Proposition \ref{thm::gausssum}, and Lemma \ref{thm::u4growth}, we can give a formula for the growth of $\widehat{\mathcal{H}}_{m,M}$ at any cusp of the congruence subgroup $\Gamma_{4M^2,M}$.

\begin{proposition}
\label{thm::cuspvalue}
Let $M,a,c\in\ZZ$ be integers such that $\gcd(a,c)=1$ and $c,M\geq1$. We set $\delta\coloneqq\gcd(c,M^2)$, set $c_0\coloneqq c/2^{\ord_2(c)}$, and set $\delta_0\coloneqq\delta/2^{\ord_2(\delta)}$. Let $k\in\ZZ$ be an integer such that $4ck/\delta\equiv1\pmod{M^2/\delta}$ if $\ord_2(c)\geq2\ord_2(M)$ and such that $ck/\delta\equiv1\pmod{4M^2/\delta}$ otherwise. Let $l\in\ZZ$ be the integer such that $4l\equiv1\pmod{c}$ when $c$ is odd. Then the following results hold:
\begin{enumerate}[leftmargin=8mm]
\item Assume that $\ord_2(c)\geq2\ord_2(M)$ holds. If $\delta\mid mM$, then we have 
$$
\nu_{\frac{a}{c}}(\widehat{\mathcal{H}}_{m,M})=-\frac{\sqrt{\delta}}{12M}\varepsilon_{\delta_0}^3\legendre{c/\delta}{M^2/\delta}\legendre{a}{\delta}e^{\frac{2\pi iakm^2}{\delta}}.
$$
Otherwise, it vanishes at the cusp $\frac{a}{c}$.
\item Assume that $0=\ord_2(c)\leq2\ord_2(M)-1$ holds. If $\delta\mid mM$, then we have 
$$
\nu_{\frac{a}{c}}(\widehat{\mathcal{H}}_{m,M})=-\frac{\sqrt{\delta}}{12M}\varepsilon_{\delta_0}^3\legendre{c/\delta}{M^2/\delta}\legendre{a}{\delta}e^{\frac{2\pi iaklm^2}{\delta}}.
$$
Otherwise, it vanishes at the cusp $\frac{a}{c}$.
\item Assume that $1\leq\ord_2(c)\leq2\ord_2(M)-1$ holds. If $\ord_2(c)=2\ord_2(M)-1$ and $\delta\parallel_2mM$, or $1\leq\ord_2(c)\leq2\ord_2(M)-2$ and $2\delta\mid mM$, then we have
$$
\nu_{\frac{a}{c}}(\widehat{\mathcal{H}}_{m,M})=-\frac{(1-i)\sqrt{\delta}}{12M}\varepsilon_{\delta_0}^3\varepsilon_{ac_0}\varepsilon\legendre{c/\delta}{M^2/\delta}\legendre{a}{\delta}e^{\frac{2\pi iakm^2}{4\delta}}.
$$
where we set
$$
\varepsilon\coloneqq
\begin{dcases}
0&\text{if }c\equiv2\pmod{4}\text{ and }m\text{ even, or }c\equiv0\pmod{4}\text{ and }m\text{ odd},\\
1&\text{otherwise},
\end{dcases}
$$
and the expression $a\parallel_2b$ means $a\mid b$ and $\ord_2(a)=\ord_2(b)$ for $a,b\in\ZZ$. Otherwise, it vanishes at the cusp $\frac{a}{c}$.
\end{enumerate}
\end{proposition}
\begin{proof}
It follows from Lemma \ref{thm::thetagrowth}, Proposition \ref{thm::gausssum}, and Lemma \ref{thm::u4growth}.
\end{proof}

Now we are ready to establish an explicit formula for the zeroth moment $H_{m,M}(n)$.

\begin{theorem}
\label{thm::quasimodularformula}
For any integers $M\geq1,m\neq0$, we have
$$
H_{m,M}(n)=\frac{2\zeta(2)}{ML(2,1_M)}\sum_{\substack{\eta\text{ primitive}\\N_{\eta}\mid M}}\sum_{\substack{d\geq1\\d\in S(\eta,m,M)}}a(\eta,d)\sigma_{1,\eta,\eta}\left(\frac{n}{d}\right)+a_{0,m,M}(n)-\lambda_{0,m,M}(n),
$$
with $S(\eta,m,M)$ defined in (\ref{eqn::drestriction}) and $a(\eta,d)$ defined in (\ref{eqn::divisorcoefficient}), where $a_{0,m,M}(n)$ is the $n$-th Fourier coefficient of a holomorphic cusp form $g_{m,M}$ of weight $2$.
\end{theorem}
\begin{proof}
We prove for any odd integer $M\geq1$. The proof for even integers $M$ is similar but more complicated, which is left to the readers. 

Set $f\coloneqq\widehat{\mathcal{H}}_{m,M}$. First, we calculate
$$
\nu_{c,\eta}(f_{\psi})\coloneqq\sum_{a\in(\ZZ/cN_{\eta}\ZZ)^{\times}}\frac{\eta(a)\nu_{\frac{a}{cN_{\eta}}}(f_{\psi})}{\phi(cN_{\eta})},
$$
where $\psi$ is a Dirichlet character modulo $M$, $\eta$ is a primitive Dirichlet character such that $\psi=\epsilon\eta$ with a primitive Dirichlet character $\epsilon$ such that $cN_{\epsilon}N_{\eta}\mid4M^2$. Because $M$ is odd, it is easy to show that $N_{\epsilon},N_{\eta}\mid M$. We set $\delta\coloneqq\gcd(cN_{\eta},M^2)$ and $\delta_0\coloneqq\delta/\delta_2$. By Proposition \ref{thm::cuspvalue}, if $\delta\nmid mM$, then we have 
$$
\nu_{c,\eta}(f_{\psi})=0,
$$
and if $\delta\mid mM$, then we have
$$
\nu_{c,\eta}(f_{\psi})=-\frac{\sqrt{\delta}\varepsilon_{\delta_0}^3}{12M}\legendre{cN_{\eta}/\delta}{M^2/\delta}\sum_{\substack{a\in(\ZZ/cN_{\eta}\ZZ)^{\times}\\x\in(\ZZ/M\ZZ)^{\times}}}\frac{\eta(a)\psi(x)}{\phi(cN_{\eta})\phi(M)}\legendre{a}{\delta}e^{\frac{2\pi iakm^2x^2}{\delta}},
$$
where $k$ is an integer such that $4cN_{\eta}k/\delta\equiv1\pmod{M^2/\delta}$. Set
$$
G(c,\eta)\coloneqq\sum_{\substack{a\in(\ZZ/cN_{\eta}\ZZ)^{\times}\\x\in(\ZZ/M\ZZ)^{\times}}}\frac{\eta(a)\psi(x)}{\phi(cN_{\eta})\phi(M)}\legendre{a}{\delta}e^{\frac{2\pi iakm^2x^2}{\delta}}.
$$
Put $c_0\coloneqq c/c_2$ and $g\coloneqq\gcd(c_0N_{\eta},km^2)$. Then by \cite[Theorem 9.12]{montgomery2007multiplicative}, we have
\begin{align*}
G(c,\eta)=&\sum_{x\in(\ZZ/M\ZZ)^{\times}}\frac{\psi(x)}{\phi(cN_{\eta})\phi(M)}\sum_{a\in(\ZZ/cN_{\eta}\ZZ)^{\times}}\eta(a)\legendre{a}{\delta}e^{\frac{2\pi iakm^2x^2}{\delta}}\\
=&\sum_{x\in(\ZZ/M\ZZ)^{\times}}\frac{\psi(x)}{\phi(M)}\overline{\chi}\left(\frac{km^2x^2}{g}\right)\chi\left(\frac{c_0N_{\eta}}{gN_{\chi}}\right)\mu\left(\frac{c_0N_{\eta}}{gN_{\chi}}\right)\frac{G(\chi)}{\phi(c_0N_{\eta}/g)}\\
=&~\delta_{\psi=\eta^2}\overline{\chi}\left(\frac{km^2}{g}\right)\chi\left(\frac{c_0N_{\eta}}{gN_{\chi}}\right)\mu\left(\frac{c_0N_{\eta}}{gN_{\chi}}\right)\frac{G(\chi)}{\phi(c_0N_{\eta}/g)},
\end{align*}
where $\chi$ is the primitive character inducing the Dirichlet character $\eta\chi_{\delta}$ and we extend the domain of the Mobius function to any positive rational numbers by setting $\mu(r)=0$ for any $r\not\in\ZZ$. 

Using multiplicativity to evaluate it, we see that
$$
G(c,\eta)=\delta(c,\eta)\delta_{\psi=\eta^2}\frac{G(\chi)}{\phi(N_{\chi})}\overline{\chi}\left(\frac{km^2}{\gcd(c_0,m^2)}\right)\chi(N_{\hat{\eta}})\prod_{\substack{p\mid N_{\hat{\eta}}\\c_p=m_p^2}}\frac{1}{1-p},
$$
where we set $\delta(c,\eta)\coloneqq1$ if all of the following conditions
\begin{enumerate}[leftmargin=8mm]
\item $c_p\mid m_p^2$ and $c_p\in\ZZ^2$ or $c_p=pm_p^2$ for any prime number $p\nmid2N_{\eta}$.
\item $c_p\mid m_p^2$ and $c_p\in\ZZ^2$ for any prime number $p\mid N_{\hat{\eta}}$.
\item $c_p=m_p^2$ for any prime number $p\mid N_{\tilde{\eta}}$.
\end{enumerate}
hold, and otherwise we set $\delta(c,\eta)\coloneqq0$. Using the fact that $4cN_{\eta}k/\delta\equiv1\pmod{M^2/\delta}$, we have
\begin{equation}
\label{eqn::nuceta}
\nu_{c,\eta}(f_{\psi})=-\frac{\delta(c,\eta)\delta_{\psi=\eta^2}G(\eta^{\star})\epsilon_{N_{\eta}}^3\sqrt{N_{\eta}}}{12M\phi(N_{\eta^{\star}})\overline{\eta^{\star}}(N_{\hat{\eta}})}\cdot\frac{\eta(4c_2C_0)\sqrt{c_0C_0}}{\phi(C_0)\eta^{\star}\big(\frac{m^2}{\gcd(c_0,m^2)}\big)}\prod_{\substack{p\mid N_{\hat{\eta}}\\c_p=m_p^2}}\frac{1}{1-p},
\end{equation}
where $C_0$ is the square-free part of $c_0$.

Applying Proposition \ref{thm::quasimodularformulageneral}, we have
\begin{equation}
\label{eqn::preformula1}
\holoproj(\widehat{\mathcal{H}}_{m,M})(\tau)=\sum_{\substack{\eta\text{ primitive}\\N_{\eta}\mid M}}\sum_{d\mid4M^2/N_{\eta}^2}a_{f_{\eta^2}}(\eta,\eta,d)E_{2,\eta,\eta}(d\tau)+g_{m,M}(\tau),
\end{equation}
with a holomorphic cusp form $g_{m,M}$ of weight $2$. Using (\ref{eqn::nuceta}), we have
\begin{equation}
\label{eqn::preformula2}
\begin{split}
a_{f_{\eta^2}}(\eta,\eta,d)=&\prod_{\substack{p\mid2M\\p\nmid N_{\eta}}}\left(1-\frac{1}{p^2}\right)^{-1}\sum_{c\mid4M^2/N_{\eta}^2}\mathcal{R}_{\eta,\eta}(d,c)\mathcal{S}_{4M^2/N_{\eta}^2,\eta,\eta}(d,c)\nu_{c,\eta}(f_{\eta^2})\\
=&-\frac{G(\eta^{\star})\epsilon_{N_{\eta}}^3\sqrt{N_{\eta}}}{12M\phi(N_{\eta^{\star}})\overline{\eta^{\star}}(N_{\hat{\eta}})}\prod_{\substack{p\mid2M\\p\nmid N_{\eta}}}\left(1-\frac{1}{p^2}\right)^{-1}s(\eta,d,m,M),
\end{split}
\end{equation}
where we set
$$
s(\eta,d,m,M)\coloneqq\sum_{c\mid4M^2/N_{\eta}^2}\mathcal{R}_{\eta,\eta}(d,c)\mathcal{S}_{4M^2/N_{\eta}^2,\eta,\eta}(d,c)\frac{\delta(c,\eta)\eta(4c_2C)\sqrt{c_0C}}{\phi(C)\eta^{\star}\big(\frac{m^2}{\gcd(c_0,m^2)}\big)}\prod_{\substack{p\mid N_{\hat{\eta}}\\c_p=m_p^2}}\frac{1}{1-p}.
$$
Using multiplicativity to evaluate the sum $s(\eta,d,m,M)$ and noting that we require $\delta=c_0N_{\eta}\mid mM$, we see that
$$
s(\eta,d,m,M)=\frac{3}{4}\cdot\frac{\delta_{d\in S(\eta,m,M)}\eta\big(\frac{-4d}{\gcd(d,m^2)}\big)}{\tilde{\eta}\big(\frac{m^2}{\gcd(d,m^2)}\big)\hat{\eta}(D_{d,m})}\sqrt{\frac{d}{D}}\prod_{\substack{p\mid N_{\hat{\eta}}\\c_p=m_p^2}}\frac{1}{1-p}\prod_{p\nmid N_{\eta}}\Psi_{d,m}(p).
$$
Combining with (\ref{eqn::quasimodulareisenstein}), (\ref{eqn::preformula1}), and (\ref{eqn::preformula2}), we see that
$$
\widehat{H}_{m,M}(n)=\frac{2\zeta(2)}{ML(2,1_M)}\sum_{\substack{\eta\text{ primitive}\\N_{\eta}\mid M}}\sum_{\substack{d\geq1\\d\in S(\eta,m,M)}}a(\eta,d)\sigma_{1,\eta,\eta}\left(\frac{n}{d}\right)+a_{0,m,M}(n),
$$
where $a_{0,m,M}(n)$ is the $n$-th Fourier coefficient of the cusp form $g_{m,M}$. Hence the desired result follows from (\ref{eqn::momentquasimodular}).
\end{proof}

We discuss some immediate applications of Theorem \ref{thm::quasimodularformula}. First, we prove an asymptotic lower bound for any even moment $H_{k,m,M}(n)$ when $\gcd(n,M)=1$.

\begin{corollary}
\label{thm::zerothmomentgrowth}
Fix integers $m\in\ZZ,M\geq1$ and an even integer $k\geq0$. Suppose that $n$ is an integer such that $\gcd(n,M)=1$. Then we have
$$
H_{k,m,M}(n)\gg_{k,M}n^{\frac{k}{2}+1},
$$
as $n\to\infty$.
\end{corollary}
\begin{proof}
By \cite[Lemma 3.7]{kane2020distribution}, for any positive real number $\epsilon>0$, we have
\begin{equation}
\label{eqn::asymptoticlowerbound}
H_{m,M}(n)\gg_{\epsilon,M}n^{1-\epsilon},
\end{equation}
as $n\to\infty$. By Theorem \ref{thm::quasimodularformula}, (\ref{eqn::momentquasimodular}), and Deligne's optimal bound \cite{deligne1974conjecture}, we see that
$$
H_{m,M}(n)=\frac{2\zeta(2)}{ML(2,1_M)}\sum_{\substack{\eta\text{ primitive}\\N_{\eta}\mid M}}\sum_{\substack{d\geq1\\d\in S(\eta,m,M)}}a(\eta,d)\eta(n/d)\sigma_1(n/d)+O(n^{\frac{1}{2}}).
$$
Noting that the summand in the inner sum vanishes unless $d=1$ because $\gcd(n,M)=1$ and comparing with the asymptotic lower bound (\ref{eqn::asymptoticlowerbound}), we have as $n\to\infty$,
$$
H_{m,M}(n)\gg_M\sigma_1(n)\geq n,
$$
because $H_{m,M}(n)$ is a positive rational number as $n\to\infty$. Using (\ref{eqn::highermomentexplicitformula}) and Deligne's optimal bound \cite{deligne1974conjecture} again, we see that as $n\to\infty$,
$$
H_{k,m,M}(n)\gg_{k,M}n^{\frac{k}{2}+1},
$$
as desired.
\end{proof}

As an application, we classify the pairs $(m,M)$ with $m\in\ZZ,M\geq1$ for which the cuspidal part $g_{m,M}$ is non-trivial. Let $S(m,M)$ be the set of primitive Dirichlet characters $\eta$ with $N_{\eta}\mid M$ satisfying the following conditions:
\begin{enumerate}[leftmargin=8mm]
\item $N_{\eta_p}\leq p^2m_p^2$ if $\eta_p$ is quadratic.
\item $\eta_p$ is trivial or quadratic unless either $p\neq2$ and $p\nmid m$, or $p=2$ and $m\equiv2\pmod{4}$ holds. 
\end{enumerate}

\begin{corollary}
\label{thm::zerocuspidalpart}
Suppose that $m,M$ are integers such that $1\leq m\leq M$. We have $g_{m,M}=0$ in Theorem \ref{thm::quasimodularformula} if and only if $1\leq M\leq 5$ or $(m,M)=(2,8),(6,8)$.
\end{corollary}
\begin{proof}
By Theorem \ref{thm::quasimodularformula}, we have
$$
H_{m,M}(1)+\lambda_{0,m,M}(1)=\widehat{H}_{m,M}(1).
$$
It is easy to verify that
$$
H_{m,M}(1)+\lambda_{0,m,M}(1)=
\begin{dcases}
1/2&\text{ if }m=M,\\
1/3&\text{ if }m=1,M-1,\\
5/12&\text{ if }m=2,M-2,\\
0&\text{ otherwise}.
\end{dcases}
$$
Thus if we can show that $\widehat{H}_{m,M}(1)-a_{0,m,M}(1)$ is positive and decays to $0$ uniformly in $m$ as $M\to\infty$, then $g_{m,M}\neq0$ for sufficiently large $M$.

Assuming that $n$ is a positive integer such that $n\equiv1\pmod{M}$ in Theorem \ref{thm::quasimodularformula}, we have
$$
\widehat{H}_{m,M}(n)-a_{0,m,M}(n)=\frac{2\zeta(2)}{ML(2,1_M)}\sum_{\eta\in S(m,M)}a(\eta,1)\sigma_1(n).
$$
By Corollary \ref{thm::zerothmomentgrowth}, we see that
$$
\frac{2\zeta(2)}{ML(2,1_M)}\sum_{\eta\in S(m,M)}a(\eta,1)
$$
is positive and it follows that $\widehat{H}_{m,M}(1)-a_{0,m,M}(1)$ is positive. Furthermore, we have
\begin{align*}
\Bigg|\sum_{\eta\in S(m,M)}a(\eta,1)\Bigg|\leq&\sum_{\eta\in S(m,M)}\frac{|\Phi_2(\eta)|}{N_{\hat{\eta}}\phi(N_{\tilde{\eta}})\sqrt{N_{\tilde{\eta}}}}\prod_{\substack{p\mid M\\p\nmid mN_{\eta}}}\left(1-\frac{1}{p^2-p}\right)\prod_{\substack{p\mid N_{\hat{\eta}}\\N_{\eta_p}^2=p^2m_p^2}}\frac{1}{1-p}\\
\ll&\sum_{\eta\pmod{M}}\frac{1}{N_{\eta}}\ll_{\epsilon}M^{\epsilon},
\end{align*}
for any positive real number $\epsilon>0$. It follows that $\widehat{H}_{m,M}(1)-a_{0,m,M}(1)$ decays to $0$ uniformly in $m$ as $M\to\infty$. Hence, we see that $g_{m,M}\neq0$ for any sufficiently large positive integer $M$. 

By running a computer program, one easily verifies that $g_{m,M}=0$ if and only if $1\leq M\leq 5$ or $(m,M)=(2,8),(6,8)$.
\end{proof}

\section{The bias in the moments of traces of Frobenius}\label{sec:EllipticBias}

In this section, we are going to use the formula for the moments of Hurwitz class numbers to investigate the bias in the moments of traces of Frobenius of elliptic curves over a finite field. Fix a prime number $p\geq5$ and integers $k,m\geq0,r,M\geq1$ such that $p\nmid M$. We have defined the moment of traces of Frobenius in (\ref{eqn::definemomenttrace}) by
$$
S_{k,m,M}(p^r)=\sum_{\substack{E/\FF_{p^r}\\\trace(E)\equiv m\pmod{M}}}\frac{\trace(E)^k}{|\aut(E)|},
$$
where the sum runs over all equivalence classes of elliptic curves over $\FF_{p^r}$. 

\subsection{Formula for the moment $S_{k,m,M}(p^r)$}

Using the work of Schoof \cite{schoof1987nonsingular}, we can relate the moment $S_{k,m,M}(p^r)$ of traces of the Frobenius to the moment $H_{k,m,M}(p^r)$ of Hurwitz class numbers in the following lemma.

\begin{lemma}
\label{thm::twomomentsrelation}
For any integers $k\geq0,r\geq1,m\in\ZZ,M\geq1$ and any prime number $p\geq5$ such that $p\nmid M$, we have
$$
2S_{k,m,M}(p^r)=H_{k,m,M}(p^r)-p^{k+1}H_{k,m\overline{p},M}(p^{r-2}),
$$
where $\overline{p}$ is any integer such that $p\overline{p}\equiv1\pmod{M}$ and the second term on the right-hand side is understood as $0$ if $r=1$.
\end{lemma}
\begin{proof}
We define 
$$
\widetilde{H}_{k,m,M}(n)\coloneqq\sum_{\substack{t\equiv m\pmod{M}\\p\nmid t}}t^kH(4n-t^2),
$$
for any integer $n$. By \cite[Lemma 2.2]{kane2020distribution}, which is essentially read off from Schoof's work \cite{schoof1987nonsingular}, we have the relation
\begin{equation}
\label{eqn::schoof}
2S_{k,m,M}(p^r)=\widetilde{H}_{k,m,M}(p^r)+E_{k,m,M}(p^r),
\end{equation}
where the error term is given by
$$
E_{k,m,M}(p^r)\coloneqq\delta_1H(4p)+\delta_2\frac{1-\chi_p(-1)}{2}+\frac{1-\chi_p(-3)}{3}p^{\frac{rk}{2}}\rho_{k,m,M}(p^r)+\frac{2^{k}(p-1)}{12}p^{\frac{rk}{2}}\rho_{k,m,M}(4p^r),
$$
where we put $\delta_1\coloneqq\delta_{m\nmid M}\delta_{k=0}\delta_{2\nmid r}$, $\delta_2\coloneqq\delta_{m\nmid M}\delta_{k=0}\delta_{2\mid r}$, and
$$
\rho_{k,m,M}(n)\coloneqq\sum_{\substack{t\equiv m\pmod{M}\\t^2=n}}\sgn(t)^k.
$$
By \cite[Lemma 4.1]{kane2020distribution} and noting that its proof indeed works for any integer $k\geq0$, we see that
$$
\widetilde{H}_{k,m,M}(p^r)=H_{k,m,M}{p^r}-H_{k,m\overline{p},M}(p^{r-2})-E_{k,m,M}(p^r),
$$
Thus, plugging it into (\ref{eqn::schoof}), we obtain
$$
2S_{k,m,M}(p^r)=H_{k,m,M}(p^r)-p^{k+1}H_{k,m\overline{p},M}(p^{r-2}),
$$
as desired.
\end{proof}

Therefore, combining this lemma with Theorem \ref{thm::quasimodularformula}, (\ref{eqn::highermomentexplicitformula}), and Theorem \ref{thm::coefficientformula}, one can deduce a formula for the moment $S_{k,m,M}(p^r)$ of the form (\ref{eqn::asymptoticexpansion}). In the following, we consider two examples $S_{2,m,M}(p)$ and $S_{2,m,M}(p^2)$.

\subsection{Bias in the second moment $S_{2,m,M}(p)$}

Because $\gcd(p,M)=1$, by Theorem \ref{thm::maintheorem2}, we have,
\begin{align*}
H_{m,M}(p)=\sum_{\eta\in S(m,M)}a(\eta,1)\eta(p)p+a_{0,m,M}(p)+\Bigg(\sum_{\eta\in S(m,M)}a(\eta,1)\eta(p)-\delta_{p,m,M}^{\star}\Bigg)\cdot1,
\end{align*}
where we define
$$
\delta_{p,m,M}^{\star}\coloneqq
\begin{dcases}
2&\text{ if }p+1\equiv m\pmod{M}\text{ and }p+1\equiv-m\pmod{M},\\
1&\text{ if }p+1\equiv m\pmod{M}\text{ or }p+1\equiv-m\pmod{M}\text{, and }m\not\equiv-m\pmod{M},\\
0&\text{ if }p+1\not\equiv\pm m\pmod{M}.
\end{dcases}
$$
Then, using (\ref{eqn::highermomentexplicitformula}) and Lemma \ref{thm::twomomentsrelation}, we see that
\begin{align*}
2S_{2,m,M}(p)=\sum_{\eta\in S(m,M)}a(\eta,1)\eta(p)\cdot p^2&+\left(\frac{a_{0,m,M}(p)}{p^{\frac{1}{2}}}+\frac{a_{2,m,M}(p)}{p^{\frac{3}{2}}}\right)\cdot p^{\frac{3}{2}}\\
&+\Bigg(\sum_{\eta\in S(m,M)}a(\eta,1)\eta(p)-\delta_{p,m,M}^{\star}\Bigg)\cdot p-O(1).
\end{align*}

Due to the Sato--Tate conjecture for Fourier coefficients of cusp forms, which is now a theorem in \cite{CHST}, it is known that the average over $p$ of the normalized Fourier coefficients vanishes because the distribution is symmetric around zero. When considering the averages, the contribution from Fourier coefficients of a fixed cusp form is thus always zero. In particular, we see that the coefficient of $p^{\frac{3}{2}}$ above vanishes on average. Therefore, we turn to study the terms of size $p$. The average of the these terms is exactly the average we considered in the introduction,
$$
\mathcal{A}_{2,m,M,1,1}=\frac{1}{2}\lim\limits_{X\to\infty}\frac{1}{\pi(X)}\sum_{p\leq X}\Bigg(\sum_{\eta\in S(m,M)}a(\eta,1)\eta(p)-\delta_{p,m,M}^{\star}\Bigg).
$$

\begin{proof}[Proof of Theorem \ref{thm:biaseval}]
Since prime numbers are uniformly distributed over congruence classes in $(\ZZ/M\ZZ)^{\times}$ for any integer $M\geq1$, we see by (\ref{eqn::divisorcoefficient}) that 
\[
\mathcal{A}_{2,m,M,1,1}=\frac{1}{2\phi(M)}\Bigg(2\prod_{q\mid M}\frac{1-\delta_{q\nmid m}/(q^2-q)}{1+1/q}-\delta_{\gcd(m-1,M)=1}-\delta_{\gcd(m+1,M)=1}\Bigg).
\]
This proves (1). 

To prove (2), it is straightforward to verify that $\mathcal{A}_{2,m,M,1,1}=0$ when $M=1$. Suppose that $M>1$ from now on and suppose to the contrary that $\mathcal{A}_{2,m,M,1,1}=0$. It is clear that
\begin{equation}
\label{eqn::positivetermbound}
0<2\prod_{q\mid M}\frac{1-\delta_{q\nmid m}/(q^2-q)}{1+1/q}<2.
\end{equation}
So the second largest lower order term is unbiased if and only if the first term in \eqref{eqn:biaseval} is equal to $1$ and exactly one of the conditions $\gcd(m-1,M)=1$ and $\gcd(m+1,M)=1$ holds. We hence conclude that 
$$
f\coloneqq2\prod_{q\mid M}\frac{1-\delta_{q\nmid m}/(q^2-q)}{1+1/q}=1.
$$
We claim that $M$ must be of the form $2^t3^s$ with $t,s\geq1$ and $6\mid m$. If $2\nmid M$ and $M$ has at least two distinct prime factors, then it is easy to check that $f$ is not an integer by inspecting the $2$-adic valuation of $f$. If $2\nmid M$ and $M=q^t$ is a prime power, then we see that
$$
f=
\begin{dcases}
\frac{2q}{q+1}\neq1&\text{ if }q\mid m,\\
\frac{2(q^2-q-1)}{q^2-1}\neq1&\text{ if }q\nmid m.
\end{dcases}
$$
Therefore we see that $2\mid M$. If $3\nmid M$, then it is each to check that $f$ is not an integer by inspecting the $3$-adic valuation of $f$. Thus, we show that $6\mid M$. In this case, we see that
$$
f=\left(1-\frac{\delta_{2\nmid m}}{2}\right)\left(1-\frac{\delta_{3\nmid m}}{6}\right)\prod_{\substack{q\mid M\\\gcd(q,6)=1}}\frac{1-\delta_{q\nmid m}/(q^2-q)}{1+1/q}.
$$
Using the assumption that $f=1$, we see that $6\mid m$ and $M$ does not have any prime factor $q\geq5$. This proves the claim. Since $M$ is of the form $2^t3^s$ with $t,s\geq1$ and $6\mid m$, it is not hard to see that $\gcd(m-1,M)>1$ and $\gcd(m+1,M)>1$. It follows that $\mathcal{A}_{2,m,M,1,1}>0$, which is a contradiction.

The assertion (3) is clear because the cuspidal part averages to zero due to \cite{CHST}, as noted above. Thus the terms of size $p^{3/2}$ averages to $0$ and $\mathcal{A}_{2,m,M,1,1}$ is indeed the second largest non-zero average.
\end{proof}

Theorem \ref{eqn:biaseval} allows us to study the sign of the bias by evaluating \eqref{eqn:biaseval}. We consider different cases and the results are summarized in the following and also in Theorem \ref{thm:biasintro}.

\begin{theorem}\label{thm:bias}
Let $m,M\in\ZZ$ be integers such that $M>1$. Then we have:
\begin{enumerate}[leftmargin=8mm]
\item If $M$ is an odd prime power $M=p^t$, then $\mathcal{A}_{2,m,M,1,1}>0$ if and only if $m\equiv\pm1\pmod{p}$.
\item If $6\mid M$, then $\mathcal{A}_{2,m,M,1,1}<0$ if and only if either $\gcd(m-1,M)=1$ or $\gcd(m+1,M)=1$.
\item We have
$$
\liminf\limits_{X\to\infty}\frac{|\left\{(m,M)|1\leq m\leq M\leq X,\mathcal{A}_{2,m,M,1,1}>0\right\}|}{|\left\{(m,M)|1\leq m\leq M\leq X\right\}|}\geq\frac{1}{4},
$$
and
$$
\liminf\limits_{X\to\infty}\frac{|\left\{(m,M)|1\leq m\leq M\leq X,\mathcal{A}_{2,m,M,1,1}<0\right\}|}{|\left\{(m,M)|1\leq m\leq M\leq X\right\}|}\geq\frac{1}{2\pi^2}.
$$
\end{enumerate}
\end{theorem}
\begin{proof}
(1) If $m\not\equiv\pm1\pmod{p}$, then we see that
$$
\delta_{\gcd(m-1,M)=1}+\delta_{\gcd(m+1,M)=1}=2.
$$
Therefore, by (\ref{eqn::positivetermbound}), it is clear that $\mathcal{A}_{2,m,M,1,1}<0$. If $m\equiv\pm1\pmod{p}$, then we see that
$$
2\prod_{q\mid M}\frac{1-\delta_{q\nmid m}/(q^2-q)}{1+1/q}=2-\frac{2p}{p^2-1}>1.
$$
Hence, we see that $\mathcal{A}_{2,m,M,1,1}>0$.

(2) In the proof of Theorem \ref{thm:biaseval} (2), we have seen that
$$
2\prod_{q\mid M}\frac{1-\delta_{q\nmid m}/(q^2-q)}{1+1/q}\leq1.
$$
If the above inequality is strict, then it is clear that $\mathcal{A}_{2,m,M,1,1}<0$ if and only if either $\gcd(m-1,M)=1$ or $\gcd(m+1,M)=1$. Otherwise, we have seen that $\gcd(m-1,M)>1$ and $\gcd(m+1,M)>1$, and in this case we have $\mathcal{A}_{2,m,M,1,1}>0$.

(3) If $M$ is even and $m$ is odd, then it is clear that $\gcd(m-1,M)>1$ and $\gcd(m+1,M)>1$. Therefore, it is easy to see that
$$
\liminf\limits_{X\to\infty}\frac{|\left\{(m,M)|1\leq m\leq M\leq X,\mathcal{A}_{2,m,M,1,1}>0\right\}|}{|\left\{(m,M)|1\leq m\leq M\leq X\right\}|}\geq\frac{1}{4}.
$$
By (2), we have
$$
|\left\{(m,M)|1\leq m\leq M\leq X,\mathcal{A}_{2,m,M,1,1}<0\right\}|\geq\sum_{\substack{M\leq X\\6\mid M}}\phi(M).
$$
To estimate it, we imitate a calculation in \cite{hardy1979introduction} as follows,
\begin{align*}
|\left\{(m,M)|1\leq m\leq M\leq X,\mathcal{A}_{2,m,M,1,1}<0\right\}|\geq&\sum_{\substack{d_1d_2\leq X\\6\mid d_1d_2}}\mu(d_1)d_2\\
=&\sum_{g\mid 6}\sum_{\substack{1\leq d_1\leq X\\\gcd(d_1,6)=g}}\mu(d_1)\sum_{\substack{1\leq d_2\leq\lfloor X/d_1\rfloor\\6\mid gd_2}}d_2\\
=&\sum_{g\mid 6}\sum_{\substack{1\leq d_1\leq X\\\gcd(d_1,6)=g}}\mu(d_1)\left(\frac{g}{12}\frac{X^2}{d_1^2}+O\left(\frac{X}{d_1}\right)\right)\\
=&~\frac{X^2}{12}\sum_{g\mid 6}\frac{\mu(g)}{g}\sum_{\substack{1\leq d_1\leq\infty\\\gcd(d_1,6)=1}}\frac{\mu(d_1)}{d_1^2}+O(X\log(X))\\
=&~\frac{X^2}{4\pi^2}+O(X\log(X)).
\end{align*}
Hence, we see that
$$
\liminf\limits_{X\to\infty}\frac{|\left\{(m,M)|1\leq m\leq M\leq X,\mathcal{A}_{2,m,M,1,1}<0\right\}|}{|\left\{(m,M)|1\leq m\leq M\leq X\right\}|}\geq\frac{1}{2\pi^2},
$$
as desired.
\end{proof}

\subsection{Bias in the second moment $S_{2,m,M}(p^2)$}

Finally, we study the bias in the average of the second moment $S_{2,m,M}(p^2)$. Again because $\gcd(p,M)=1$, Theorem \ref{thm::maintheorem2} implies that
\begin{align*}
H_{m,M}(p^2)=\sum_{\eta\in S(m,M)}a(\eta,1)\eta^2(p)p^2+a_{0,m,M}(p^2)+\Bigg(\sum_{\eta\in S(m,M)}a(\eta,1)\eta^2(p)-\delta_{p^2,m,M}^{\star}\Bigg)\cdot p+O(1),
\end{align*}
where we define
$$
\delta_{p^2,m,M}^{\star}\coloneqq
\begin{dcases}
2&\text{ if }2p\equiv m\pmod{M}\text{ and }2p\equiv-m\pmod{M},\\
1&\text{ if }2p\equiv m\pmod{M}\text{ or }2p\equiv-m\pmod{M}\text{, and }m\not\equiv-m\pmod{M},\\
0&\text{ if }2p\not\equiv\pm m\pmod{M}.
\end{dcases}
$$
Then, using (\ref{eqn::highermomentexplicitformula}) and Lemma \ref{thm::twomomentsrelation}, we see that
\begin{align*}
2S_{2,m,M}(p^2)=\sum_{\eta\in S(m,M)}&a(\eta,1)\eta^2(p)\cdot p^4+(a_{0,m,M}(p^2)p^2+a_{2,m,M}(p^2))\\
&+\Bigg(\sum_{\eta\in S(m,M)}a(\eta,1)\eta^2(p)-\delta_{p^2,m,M}^{\star}-H_{0,\overline{p}m,M}(1)\Bigg)\cdot p^3+O(p^2)
\end{align*}

We again know that the contribution to the averages from the cuspidal part is zero due to \cite{CHST}. Thus we consider the average of the terms of size $p^3$, which is given by
$$
\mathcal{A}_{2,m,M,2,3}=\frac{1}{2}\lim\limits_{X\to\infty}\frac{1}{\pi(X)}\sum_{p\leq X}\Bigg(\sum_{\eta\in S(m,M)}a(\eta,1)\eta^2(p)-\delta_{p^2,m,M}^{\star}-H_{0,\overline{p}m,M}(1)\Bigg).
$$

By definition, it is easy to verify that if $M\geq3$, then we have
$$
H_{2,m,M}(1)=
\begin{dcases}
1/3&\text{ if }m\equiv\pm1\pmod{M},\\
-1/3&\text{ if }m\equiv\pm2\pmod{M},\\	
0&\text{ otherwise}.
\end{dcases}
$$
By (\ref{eqn::divisorcoefficient}), if $M\geq3$, we have
\begin{equation}
\label{eqn::biasformula2}
\mathcal{A}_{2,m,M,2,3}=\frac{1}{2\phi(M)}\Bigg(2\prod_{q\mid M}\frac{q^2-q-1}{q^2-1}\sum_{\substack{\eta\in S(m,M)\\N_{\eta^2}=1}}\frac{(-1)^{\omega}\epsilon_{N_{\eta_0}}^3\Phi_2(\eta)}{G(\eta)\sqrt{N_{\eta}}}\prod_{q\mid mN_{\eta}}\frac{q}{q^2-q-1}+\epsilon_{m,M}\Bigg),
\end{equation}
where $\omega$ is the number of distinct prime factors of $N_{\eta}$ and we set
$$
\epsilon_{m,M}\coloneqq
\begin{dcases}
-2&\text{ if }M\equiv1\pmod{2}\text{ and }\gcd(m,M)=1,\\
-4/3&\text{ if }M\equiv0\pmod{2},m\equiv0\pmod{2},\text{ and }\gcd(m/2,M/2)=1,\\
-2/3&\text{ if }M\equiv0\pmod{2},m\equiv1\pmod{2},\text{ and }\gcd(m,M)=1,\\
0&\text{ otherwise}.
\end{dcases}
$$

With this formula, it is not hard to determine the sign of the bias.

\begin{theorem}
\label{thm::lasttheorem}
Let $m,M\in\ZZ$ be integers such that $M\geq3$ and $4\nmid M$. Then $\mathcal{A}_{2,m,M,2,3}<0$ if and only if $\gcd(m,M_{\operatorname{odd}})=1$.
\end{theorem}
\begin{proof}
If $\gcd(m,M_{\operatorname{odd}})>1$, then we see that $\gcd(m,M)$ is odd. It follows that $\epsilon_{m,M}=0$ and thus $\mathcal{A}_{2,m,M,2,3}>0$. From now on, we assume that $\gcd(m,M_{\operatorname{odd}})=1$. Assume that $M\equiv1\pmod{2}$. Then we have $\epsilon_{m,M}=-2$. Furthermore using multiplicativity, it is easy to see that
\begin{align*}
2\prod_{q\mid M}\frac{q^2-q-1}{q^2-1}\sum_{\substack{\eta\in S(m,M)\\N_{\eta^2}=1}}\frac{(-1)^{\omega}\epsilon_{N_{\eta_0}}^3\Phi_2(\eta)}{G(\eta)\sqrt{N_{\eta}}}\prod_{q\mid mN_{\eta}}\frac{q}{q^2-q-1}=2\prod_{q\mid M}\frac{q^2-q-1+(-1)^{\frac{q+1}{2}}}{q^2-1}<2.
\end{align*}
Therefore we see that $\mathcal{A}_{2,m,M,2,3}<0$ in this case. Assume that $M\equiv2\pmod{4}$. Then we have
$$
\epsilon_{m,M}=
\begin{dcases}
-4/3&\text{ if }m\equiv0\pmod{2},\\
-2/3&\text{ if }m\equiv1\pmod{2}.
\end{dcases}
$$
The calculation of the positive term is similar to the odd cases. Thus, it is easy to see that $\mathcal{A}_{2,m,M,2,3}<0$ in this case. This finishes the proof.
\end{proof}

\bibliographystyle{plain}
\bibliography{reference}

\begin{thebibliography}{10}

\bibitem{Asada}
Megumi Asada, Ryan~C Chen, Eva Fourakis, Yujin~Hong Kim, Andrew Kwon,
  Jared~Duker Lichtman, Blake Mackall, Steven~J Miller, Eric Winsor, Karl
  Winsor, et~al.
\newblock Lower-order biases in the second moment of dirichlet coefficients in
  families of l-functions.
\newblock {\em Experimental Mathematics}, 32(3):431--456, 2023.

\bibitem{aygin2022projections}
Zafer~Selcuk Aygin.
\newblock Projections of modular forms on {E}isenstein series and its
  application to {S}iegel’s formula.
\newblock {\em The Ramanujan Journal}, 57(4):1223--1252, 2022.

\bibitem{bringmann2017harmonic}
Kathrin Bringmann, Amanda Folsom, Ken Ono, and Larry Rolen.
\newblock {\em Harmonic {Maass} forms and mock modular forms: theory and
  applications}, volume~64.
\newblock American Mathematical Soc., 2017.

\bibitem{bringmann2019sums}
Kathrin Bringmann and Ben Kane.
\newblock Sums of class numbers and mixed mock modular forms.
\newblock {\em Mathematical Proceedings of the Cambridge Philosophical
  Society}, 167(2):321--333, 2019.

\bibitem{bringmann2022conjectures}
Kathrin Bringmann and Ben Kane.
\newblock Conjectures of {S}un about sums of polygonal numbers.
\newblock {\em La Matematica}, 1(4):809--828, 2022.

\bibitem{bringmann2021odd}
Kathrin Bringmann, Ben Kane, and Sudhir Pujahari.
\newblock Odd moments for the trace of {Frobenius and the Sato--Tate}
  conjecture in arithmetic progressions.
\newblock {\em arXiv preprint arXiv:2112.08205}, 2021.

\bibitem{bringmann2023formulas}
Kathrin Bringmann, Ben Kane, and Sudhir Pujahari.
\newblock Formulas for moments of class numbers in arithmetic progressions.
\newblock {\em Acta Arithmetica}, 207:19--38, 2023.

\bibitem{brown2008elliptic}
Brittany Brown, Neil~J Calkin, Timothy~B Flowers, Kevin James, Ethan Smith, and
  Amy Stout.
\newblock Elliptic curves, modular forms, and sums of {H}urwitz class numbers.
\newblock {\em Journal of Number Theory}, 128(6):1847--1863, 2008.

\bibitem{cho2018number}
Bumkyu Cho.
\newblock On the number of representations of integers by quadratic forms with
  congruence conditions.
\newblock {\em Journal of Mathematical Analysis and Applications},
  462(1):999--1013, 2018.

\bibitem{CHST}
L.~Clozel, M.~Harris, N.~Shepherd-Barron, and R.~Taylor.
\newblock Automorphy for some $\ell$-adic lifts of automorphic mod $\ell$
  {G}alois representations.
\newblock {\em Publ. Math. Inst. Hautes Études Sci.}, 108:1--181, 2008.

\bibitem{cohen1975sums}
Henri Cohen.
\newblock Sums involving the values at negative integers of {$L$-functions} of
  quadratic characters.
\newblock {\em Mathematische Annalen}, 217(3):271--285, 1975.

\bibitem{deligne1974conjecture}
Pierre Deligne.
\newblock La conjecture de {Weil}. {I}.
\newblock {\em Publications Math{\'e}matiques de l'Institut des Hautes
  {\'E}tudes Scientifiques}, 43(1):273--307, 1974.

\bibitem{gierster1883ueber}
Joseph Gierster.
\newblock Ueber {R}elationen zwischen {K}lassenzahlen bin{\"a}rer quadratischer
  {F}ormen von negativer {D}eterminante.
\newblock {\em Mathematische Annalen}, 21(1):1--50, 1883.

\bibitem{gross1986heegner}
BH~Gross and DB~Zagier.
\newblock Heegner points and derivatives of {$L$}-series.
\newblock {\em Inventiones mathematicae}, 84:225--320, 1986.

\bibitem{hardy1979introduction}
Godfrey~Harold Hardy and Edward~Maitland Wright.
\newblock {\em An introduction to the theory of numbers}.
\newblock Oxford university press, 1979.

\bibitem{hurwitz1885ueber}
Adolf Hurwitz.
\newblock Ueber {R}elationen zwischen {C}lassenanzahlen bin{\"a}rer
  quadratischer {F}ormen von negativer {D}eterminante.
\newblock {\em Mathematische Annalen}, 25(2):157--196, 1885.

\bibitem{kane2020distribution}
Ben Kane and Sudhir Pujahari.
\newblock Distribution of moments of hurwitz class numbers in arithmetic
  progressions and holomorphic projection.
\newblock {\em Transactions of the American Mathematical Society},
  376(08):5503--5519, 2023.

\bibitem{KS2}
Nicholas~M Katz and Peter Sarnak.
\newblock Zeros of zeta functions and symmetries.
\newblock {\em Bull. AMS}, 36:1--26, 1999.

\bibitem{KS1}
Nicholas~M Katz and Peter Sarnak.
\newblock {\em Random matrices, {F}robenius eigenvalues, and monodromy},
  volume~45.
\newblock American Mathematical Society, 2023.

\bibitem{KazalickiNaskrecki}
Matija Kazalicki and Bartosz Naskr{\k{e}}cki.
\newblock Second moments and the bias conjecture for the family of cubic
  pencils.
\newblock {\em arXiv preprint arXiv:2012:11306}, 2021.

\bibitem{kronecker1860ueber}
L~Kronecker.
\newblock Ueber die {A}nzahl der verschiedenen {C}lassen quadratischer {F}ormen
  von negativer {D}eterminante.
\newblock {\em Journal f{\"u}r die reine und angewandte Mathematik},
  57:248--255, 1860.

\bibitem{mertens2016eichler}
Michael~H Mertens.
\newblock {Eichler--Selberg} type identities for mixed mock modular forms.
\newblock {\em Advances in Mathematics}, 301:359--382, 2016.

\bibitem{Michel}
Philippe Michel.
\newblock Rang moyen de familles de courbes elliptiques et lois de sato-tate.
\newblock {\em Monatshefte f{\"u}r Mathematik}, 120:127--136, 1995.

\bibitem{Mlr1}
Steven~J. Miller.
\newblock One- and two-level densities for families of elliptic curves:
  evidence for the underlying group symmetries.
\newblock {\em PhD thesis}, 2002.

\bibitem{Mlr2}
Steven~J Miller.
\newblock One-and two-level densities for rational families of elliptic curves:
  evidence for the underlying group symmetries.
\newblock {\em Compositio Mathematica}, 140(4):952--992, 2004.

\bibitem{MW}
Steven~J. Miller and Yan Weng.
\newblock Biases in moments of the dirichlet coefficients in one-and
  two-parameter families of elliptic curves.
\newblock {\em arXiv preprint arXiv:2103.03942}, 2021.

\bibitem{montgomery2007multiplicative}
Hugh~L Montgomery and Robert~C Vaughan.
\newblock {\em Multiplicative number theory I: {C}lassical theory}.
\newblock Number~97 in Cambridge Studies in Advanced Mathematics. Cambridge
  university press, 2007.

\bibitem{ono2021distribution}
Ken Ono, Hasan Saad, and Neelam Saikia.
\newblock Distribution of values of {G}aussian hypergeometric functions.
\newblock {\em Pure and Applied Mathematics Quarterly}, 19(1):371--407, 2023.

\bibitem{schoeneberg2012elliptic}
Bruno Schoeneberg.
\newblock {\em Elliptic modular functions: an introduction}, volume 203.
\newblock Springer Science \& Business Media, 2012.

\bibitem{schoof1987nonsingular}
Ren{\'e} Schoof.
\newblock Nonsingular plane cubic curves over finite fields.
\newblock {\em Journal of combinatorial theory, Series A}, 46(2):183--211,
  1987.

\bibitem{shimura1973modular}
Goro Shimura.
\newblock On modular forms of half integral weight.
\newblock {\em Annals of Mathematics}, 97(3):440--481, 1973.

\bibitem{zagier1974nombres}
Don Zagier.
\newblock Nombres de classes et formes modulaires de poids 3/2.
\newblock {\em S{\'e}minaire de Th{\'e}orie des Nombres de Bordeaux}, 4:1--2,
  1974.

\bibitem{zindulka2024sums}
Mikul{\'a}{\v{s}} Zindulka.
\newblock Sums of hurwitz class numbers, cm modular forms, and primes of the
  form $ x^2+ ny^2$.
\newblock {\em arXiv preprint arXiv:2405.07565}, 2024.

\end{thebibliography}

\end{document}